\documentclass[a4paper]{article}

\usepackage{mystyle}

\title{Spherical $\mathrm{CR}$ Dehn Surgeries}
\author{Miguel Acosta}
\begin{document}

\maketitle
\begin{abstract}

Consider a three dimensional cusped spherical $\mathrm{CR}$ manifold $M$ and suppose that the holonomy representation of $\pi_1(M)$ can be deformed in such a way that the peripheral holonomy is generated by a non-parabolic element. We prove that, in this case, there is a spherical $\mathrm{CR}$ structure on some Dehn surgeries of $M$. The result is very similar to R. Schwartz's spherical $\mathrm{CR}$ Dehn surgery theorem, but has weaker hypotheses and does not give the unifomizability of the structure. We apply our theorem in the case of the Deraux-Falbel structure on the Figure Eight knot complement and obtain spherical $\mathrm{CR}$ structures on all Dehn surgeries of slope $-3 + r$ for $r \in \mathbb{Q}^{+}$ small enough.

\end{abstract}
\section{Introduction}

The celebrated theorem of hyperbolic Dehn surgeries of Thurston, stated in \cite{gt3m}, says that all but a finite number of Dehn surgeries of a one cusped hyperbolic manifold $M$ admit complete hyperbolic structures with developing maps and holonomy representations close to those of $M$. The same question arises for other geometric structures. We focus here on \CR {} structures i.e. structures modeled on the boundary at infinity of the complex hyperbolic  plane with group of automorphisms $\pu21$. In his book (\cite{schwartz}), Schwartz shows a \CR {} Dehn surgery theorem that gives, under some convergence hypotheses, uniformizable \CR {} structures on some Dehn surgeries on a cusped \CR {} manifold. In this paper, we prove a similar theorem using techniques coming from $(G,X)$-structures and the geometry of $\dh2c$ instead of the approach of discreteness of group representations and actions on $\h2c$. Theorem \ref{thm_chirurgie} has weaker hypotheses than Schwartz' theorem, but we obtain geometric structures on the surgeries for which we do not know whether they are uniformizable or not.
 
 For the reader, the example to keep in mind, treated in section \ref{section_deraux-falbel}, is the Deraux-Falbel structure on the figure eight knot complement constructed in \cite{falbel}. For this example, Deraux shows in \cite{deraux_uniformizations} that there is a one parameter family of \CR {}  uniformizations on the figure eight knot complement with parabolic peripheral holonomy. In \cite{character_sl3c}, Falbel, Guilloux, Koseleff, Rouillier and Thistlethwaite describe the $\mathrm{SL}_3(\mathbb{C})$-character variety of the fundamental group of the figure eight knot. They give an explicit parametrization for the component in $\su21$ containing the holonomy representation of the Deraux-Falbel structure. This component also gives rise to \CR {} structures near the Deraux-Falbel structure. With this parametrization and theorem \ref{thm_chirurgie}, we obtain the following proposition:
 
 \begin{thm*}
Let $M$ be the figure eight knot complement. For the usual \footnote{For us, the usual marking is the one given by Thurston in \cite{gt3m}. 
} marking of the peripheral torus of $M$:
 \begin{enumerate}
\item There exist infinitely many  \CR {} structures on the Dehn surgery of $M$ of slope $-3$.
\item There exists $\delta > 0$ such that for all $r \in \mathbb{Q} \cap (0, \delta )$, there is a  \CR {} structure on the Dehn surgery of $M$ of slope $-3 + r$.
   \end{enumerate}
\end{thm*}

 In section \ref{section_h2c}, we recall some properties about $\h2c$, $\dh2c$ and $\pu21$ and set some notation. We look in detail at the dynamics of one parameter subgroups of $\pu21$ acting on $\dh2c$. Understanding these dynamics will be crucial to show the surgery theorem.
 Section \ref{section_surgeries} deals with deformation of $(G,X)$ structures, fixes some notation and a marking of a peripheral torus in order to state the main theorem of this paper (\ref{thm_chirurgie}). In section \ref{section_deraux-falbel},   we apply theorem \ref{thm_chirurgie} in the case of the Deraux-Falbel structure, by checking the hypotheses and looking at the deformation space as given in \cite{character_sl3c}. Finally, in section \ref{section_proof}, we give a complete proof of the surgery theorem.
\section{Generalities on $\h2c$ and its isometries} \label{section_h2c}
In this section we recall some facts about the hyperbolic complex plane $\h2c$ and its boundary at infinity $\dh2c$ and set notation for them. We study the group of holomorphic isometries of $\h2c$, identified to $\pu21$, by describing its one parameter subgroups. Almost all stated results can be found in the thesis of Genzmer \cite{genzmer} and in the book of Goldman \cite{goldman}.

\subsection{The space $\h2c$ and its boundary at infinity.}

 We begin by giving a construction of the hyperbolic complex plane. Let $V$ be a complex vector space of dimension 3 endowed with a Hermitian product  $\langle \cdot, \cdot \rangle$. Denote by $\Phi$ the associated Hermitian form. We suppose that $\Phi$ has signature $(2,1)$. By setting $V_{-} = \{v \in V\setminus \{0\} \mid \Phi(v) < 0 \}$ , $V_{0} = \{v \in V\setminus \{0\} \mid \Phi(v) = 0 \}$ and $V_{+} = \{v \in V\setminus \{0\} \mid \Phi(v) > 0 \}$, the complex hyperbolic plane is defined as $\mathbb{P}V_-$, endowed with the Hermitian metric induced by $\Phi$, and its boundary at infinity $\dh2c$ as $\mathbb{P}V_0$.
 
\begin{notat}
 We will use several times projectivizations of vector spaces and of groups of matrices.  In this case, we will denote with usual parentheses $"("$ and $")"$ the objects before projectivization and with square brackets $"["$ and $"]"$ the class of an object in the projectivized space. For example, if $Z \in \mathbb{C}^3 \setminus \{0\}$, then $[Z] \in \cp2$ is the projection of $Z$.
\end{notat}

From now on, we will use two different models of the complex hyperbolic plane, going from one to another by a conjugation. In both cases, the vector space is $V=\mathbb{C}^3$.

 \begin{notat}
  Let 
\[J_1 = \begin{pmatrix}
1 & 0 & 0 \\
0 & 1 & 0 \\
0 & 0 & -1 \\
\end{pmatrix}  
\text{ and }
 J_2 = \begin{pmatrix}
0 & 0 & 1 \\
0 & 1 & 0 \\
1 & 0 & 0 \\
\end{pmatrix} . \]

There are the matrices of the Hermitian products $\langle \cdot , \cdot \rangle_1$ and $\langle \cdot , \cdot \rangle_2$ given, for $W= \begin{pmatrix}W_1\\W_2\\W_3\end{pmatrix}$ and $Z=\begin{pmatrix}Z_1\\Z_2\\Z_3\end{pmatrix}$ in $\mathbb{C}^3$, by:

\[ \langle W , Z \rangle_1 ={}^t\con{W}J_1Z= \con{W_1}Z_1 + \con{W_2}Z_2 - \con{W_3}Z_3 \]

\[
 \langle W , Z \rangle_2 ={}^t\con{W}J_2Z= \con{W_1}Z_3 + \con{W_2}Z_2 + \con{W_3}Z_1
\]
  
 \end{notat}

These two Hermitian forms are conjugate by Cayley's matrix $C= \frac{1}{\sqrt{2}}  \begin{pmatrix}
1 & 0 & 1 \\
0 & \sqrt{2} & 0 \\
1 & 0 & -1 \\
\end{pmatrix}$, that satisfies $C^{-1} = C^* = C$
 \begin{defn}
  By identifying $V$ to $\mathbb{C}^3$ and $\langle \cdot , \cdot \rangle$ to $\langle \cdot , \cdot \rangle_1$, we obtain the \emph{ball model}. We then have :
  \[\h2c = \left\{ \begin{bmatrix}Z_1\\Z_2\\1\end{bmatrix} \in \cp2 \mid |Z_1|^2 + |Z_2|^2 < 1 
  \right\} \]
  
  \[\text{ and } \dh2c = \left\{ \begin{bmatrix}Z_1\\Z_2\\1\end{bmatrix} \in \cp2 \mid |Z_1|^2 + |Z_2|^2 = 1 
  \right\}\]
 \end{defn} 
 
 With this model, we see that $\h2c$ is homeomorphic to the ball $B^4$ and $\dh2c$ is homeomorphic to the sphere $S^3$. Let's see the other model that we will consider, the Siegel model.
 
 \begin{defn}
  By identifying $V$ to $\mathbb{C}^3$ and $\langle \cdot , \cdot \rangle$ to $\langle \cdot , \cdot \rangle_2$, we obtain the \emph{Siegel model}. It is given by:
 \[\h2c = \left\{ \begin{bmatrix}Z_1\\Z_2\\1\end{bmatrix} \in \cp2 \mid 2 \Re (Z_1) + |Z_2|^2 < 0 
  \right\} \subset \cp2 \]
  
  \[\text{ and } \dh2c = \left\{ \begin{bmatrix}-\frac{1}{2}(|z|^2+it)\\z\\1\end{bmatrix} \mid (z,t)\in \mathbb{C}\times \mathbb{R} 
  \right\} \cup \left\{\begin{bmatrix}1\\0\\0\end{bmatrix} \right\}\]  
  
 \end{defn}
 
 With this model, we can identify $\dh2c$ to $(\mathbb{C}\times \mathbb{R}) \cup \{\infty\}$. Removing the point at infinity, we obtain the Heisenberg group, defined as $\mathbb{C} \times \mathbb{R}$ with multiplication $(w,s)*(z,t) = (w+z,s+t+2\Im(w \con{z}))$.

We are going to use
 complex geodesics, which are intersections of complex lines of $\mathbb{P}V$ with $\h2c$, and their boundaries at infinity, called $\mathbb{C}$-circles.

\subsection{Holomorphic isometries of $\h2c$ and invariant flows}

 We defined above the complex hyperbolic space and have seen two of its models. The group of holomorphic isometries of this space is $\pu21$, as described below. It acts transitively on the unit tangent fiber bundle of $\h2c$.
  
\begin{notat}
 Let $\mathrm{U}(2,1)$ be the group of matrices matrices of $\mathrm{GL}_3(\mathbb{C})$ such that $A^*JA = J$ for $J=J_1$ or $J_2$ (according to the model in which we work). Let $\su21$ be the subgroup of matrices of determinant 1 and $\pu21$ its projectivization.
\end{notat}
  
\begin{rem}
 Given $[A]\in \pu21$, there are exactly three lifts of $[A]$ in $\su21$, which are $A$, $\omega A$ and $\omega^2A$, where $\omega$ is a cube root of 1.
\end{rem}

 We state in detail a classification of the elements of $\pu21$. We use the notations and state the results of chapter one of Genzmer's thesis \cite{genzmer}. Each isometry of $\h2c$ extends continuously to $\h2c \cup \dh2c$, which is compact. By Brouwer's theorem, it has fixed points. Isometries are classified by their fixed points in  $\h2c \cup \dh2c$.

\begin{defn}
 An isometry $g \neq id$ of $\h2c$ is called :
 \begin{itemize} 
  \item \emph{elliptic} if it has at least one fixed point in $\h2c$.
  \item \emph{parabolic} if it is not elliptic and has exactly one fixed point in $\dh2c$
  \item \emph{loxodromic} if it is not elliptic and has exactly two fixed points in $\dh2c$. 
  
  \end{itemize}
\end{defn}

 We can state this classification in terms of eigenvalues. The eigenvalues of an element of $\pu21$ are only defined up to multiplication by $\omega$; we give a condition on the eigenvalues of a lift in $\su21$

\begin{prop}
 Let $U \in \su21 \setminus \{\mathrm{Id}\}$. Then $U$ is in one of the three following cases:
 
 \begin{enumerate}
  \item $U$ has an eigenvalue $\lambda$ of modulus different from 1. Then $[U]$ is loxodromic.
  \item $U$ has an eigenvector $v \in V_{-}$. Then $[U]$ is elliptic and its eigenvalues have modulus equal to 1 but are not all equal.
  \item all eigenvalues of $U$ have modulus 1 and $G$ has an eigenvector $v \in V_{0}$. Then $[U]$ is parabolic.
 \end{enumerate}
\end{prop}

To refine on this classification, we will consider different cases when there are double eigenvalues. We give the following definition:

\begin{defn} \label{def_regulier}
 Let $U \in \su21 \setminus \{\mathrm{Id}\}$. We say that $U$ is
 
 \begin{enumerate}
  \item \textit{regular} if its three eigenvalues are different.
  \item \textit{unipotent} if its three eigenvalues are equal (and so equal to a cube root of 1).
\end{enumerate}  
\end{defn}

The definition extends to $\pu21$ ; we will speak of regular elements of $\pu21$. In that case the eigenvalues are well defined up to multiplication by $\omega$. Thanks to the following remark, we know that regular elements are easier to manipulate : 

\begin{rem}
 Let $[U]\in \pu21$ be a regular element. Then
 $[U]$ is determined by its three eigenvalues $\alpha, \beta , \gamma$ and its three fixed points $[u],[v],[w]$ in $\cp2$.
\end{rem}

It is possible to know if an element is regular only by knowing its trace. An element of $\su21$ is regular if and only if its characteristic polynomial has no double root. But, if $U \in \su21$ and $z= \mathrm{tr}(U)$, the characteristic polynomial of $U$ is $\chi_U= X^3 - zX^2 + \overline{z}X - 1$. We only need then to compute the resultant of $\chi_U$ and $\chi_U'$. We get the next proposition, that can be found in Goldman's book \cite{goldman}.

 \begin{prop}\label{fonction_goldman}
  For $z \in  \mathbb{C}$, let $f(z)=|z|^4-8\Re (z^3) + 18|z|^2-27$. Let $U \in \su21$. Then $U$ is regular if and only if $f(\mathrm{tr}(U))\neq 0$. Furthermore, if $f(\mathrm{tr}(U))< 0$ then $[U]$ is regular elliptic and if $f(\mathrm{tr}(U))>0$ then $[U]$ is loxodromic. 
 \end{prop}

\begin{rem} It is suitable to make two remarks about the proposition:
 \begin{enumerate}
  \item $f(z)=f(\omega z)$. Therefore we can define the function $f \circ \mathrm{tr}$ on $\pu21$.
  \item For a parabolic element $[U]$, the equality $f(\mathrm{tr}(U)) = 0$ holds, but there are nonregular elliptic elements for which $f(\mathrm{tr}(U)) = 0$.
 \end{enumerate}
\end{rem}

In order to study \CR {} structures and their surgeries, we will use the flows of vector fields associated to some elements of $\pu21$. The geometric objects that we are going to consider are invariant vector fields induced by elements of $\pu21$. Let's begin by looking at an infinitesimal level: an element of the Lie algebra $\mathfrak{su}(2,1)$ defines a vector field on $\dh2c$ invariant under its exponential map.

\begin{notat}
Let $X \in \mathfrak{su}(2,1)$. It defines a vector field on $\dh2c$ invariant by $\exp(X)$ given at a point $x$ by :
\[\left. \frac{\mathrm{d}}{\mathrm{d}t} \right|_{t=0} \exp(tX)\cdot x.\]

Let $\phi_t^X$ be the flow of this vector field, so $\phi_t^X(x) = \exp(tX)\cdot x$. If there is no ambiguity for $X$, we will only write $\phi_t$.
\end{notat}

\begin{rem}
 If $[U]\in \pu21$ is close enough to a unipotent element, it defines a vector field on $\dh2c$. Indeed, possibly after changing the lift, we can suppose that the eigenvalues of $U$ are near 1, and consider the vector field associated to $\Log(U)$. Then, $\phi^{\Log(U)}_1$ has the same action as $[U]$.
\end{rem}

\subsection{Description of isometries and invariant flows}

We are going to describe briefly the elements of $\pu21$, and classify them by their type and the dynamics of their action on $\cp2$.
 We are going to study the dynamics of some flows of the form $\phi_t^{\Log(U)}$, where $U$ is close to a unipotent element. We describe here flows associated to regular elliptic, loxodromic and parabolic elements.

\subsubsection{Regular elliptic flows}
 Consider a regular elliptic element in $\su21$ in the ball model. Perhaps after a conjugation, we can suppose that it is equal to
 
  \[E_{\alpha,\beta,\gamma} = \begin{pmatrix}
e^{i\alpha} & 0 & 0 \\
0 & e^{i\beta} & 0 \\
0 & 0 &  e^{i\gamma} \\
\end{pmatrix} .\]

We will also suppose that $\alpha$, $\beta$ and $\gamma$ are not all equal to zero and are small enough. In this case, 
$\gamma = -\alpha - \beta$ and

\[\Log(E_{\alpha,\beta,\gamma}) = \begin{pmatrix}
i\alpha & 0 & 0 \\
0 & i\beta & 0 \\
0 & 0 &  i\gamma \\
\end{pmatrix} .\]

The flow of the associated vector field acts on $\dh2c$ by:
\[\phi_t^{\Log(E_{\alpha,\beta,\gamma})} \begin{bmatrix}
Z_1 \\ Z_2 \\ 1
\end{bmatrix}
=
\begin{bmatrix}
e^{it(\alpha-\gamma)}Z_1 \\ e^{it(\beta-\gamma)}Z_2 \\ 1
\end{bmatrix} 
=
\begin{bmatrix}
e^{it(2\alpha+\beta)}Z_1 \\ e^{it(2\beta+\alpha)}Z_2 \\ 1
\end{bmatrix} \]

\begin{rem}
The flow stabilizes the two $\mathbb{C}$-circles $C_1=[e_1]^{\perp} \cap \dh2c = \left\{ \begin{bmatrix}
0 \\ e^{i\theta} \\ 1
\end{bmatrix} \mid \theta \in \mathbb{R}\right\}$ and $C_2 = [e_2]^{\perp} \cap \dh2c = \left\{ \begin{bmatrix}
 e^{i\theta} \\ 0 \\ 1
\end{bmatrix} \mid \theta \in \mathbb{R}\right\}$ on which it acts as rotations by angles $2 \beta + \alpha$ and $2\alpha + \beta$ respectively.

\begin{figure}[H]
\center
 \includegraphics[width=5cm]{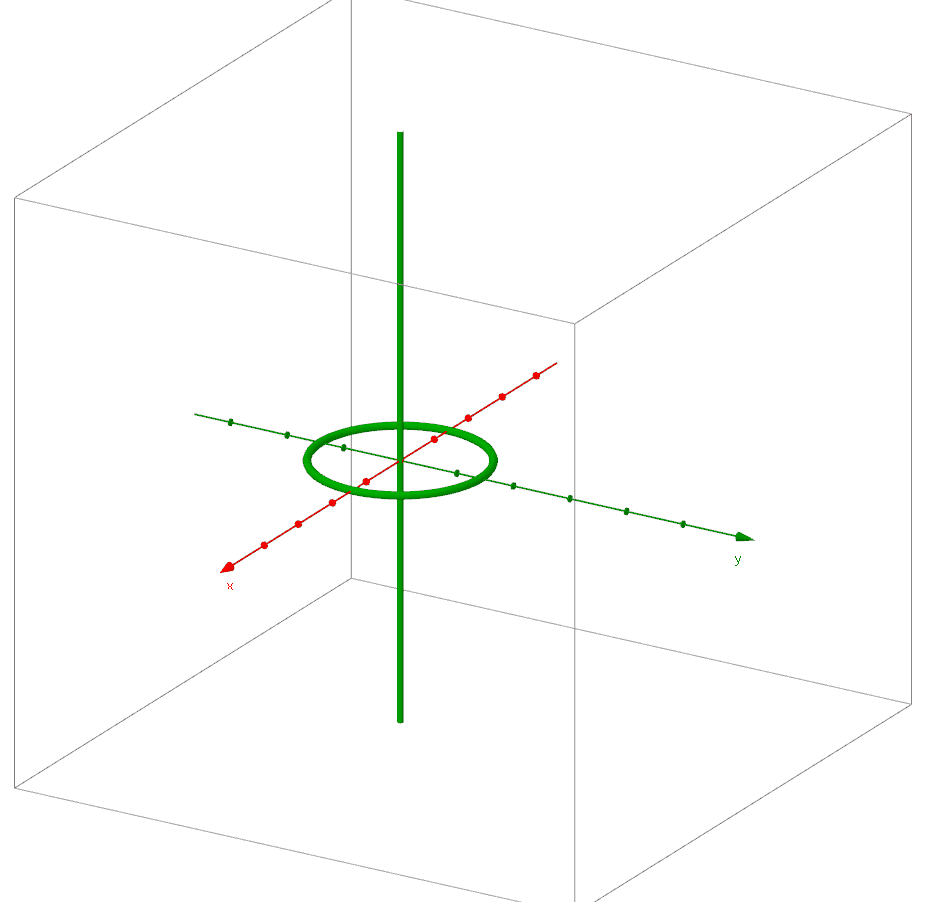}
 \caption{Invariant $\mathbb{C}$-circles for a regular elliptic flow in Siegel's model.}
\end{figure}
\end{rem}

\begin{rem}
 The centralizer of $E_{\alpha,\beta,\gamma}$ is $C(E_{\alpha,\beta,\gamma})= \left\{E_{\theta_1,\theta_2,-(\theta_1+\theta_2)} \mid (\theta_1, \theta_2) \in \mathbb{R}^2 \right\}$. The orbits of this subgroup in $\dh2c$ are $C_1$, $C_2$ and the subsets  $T_r$ for $r\in ]0,1[$ deined by
 
 \[T_r = \left\{ \begin{bmatrix}
 Z_1 \\ Z_2 \\ 1
\end{bmatrix} \in \dh2c \mid |Z_2|=r , |Z_1| = \sqrt{1-r^2}  \right\}.\] 

 The orbits $T_r$ are embedded tori in $\dh2c$ with core curves $C_1$ and $C_2$. They are all invariant under the action of $\phi_t^{\Log(E_{\alpha,\beta, \gamma})}$. We can see an example in figure \ref{tore_inv}.
\end{rem}

\begin{figure}[H]
\center
 \includegraphics[width=5cm]{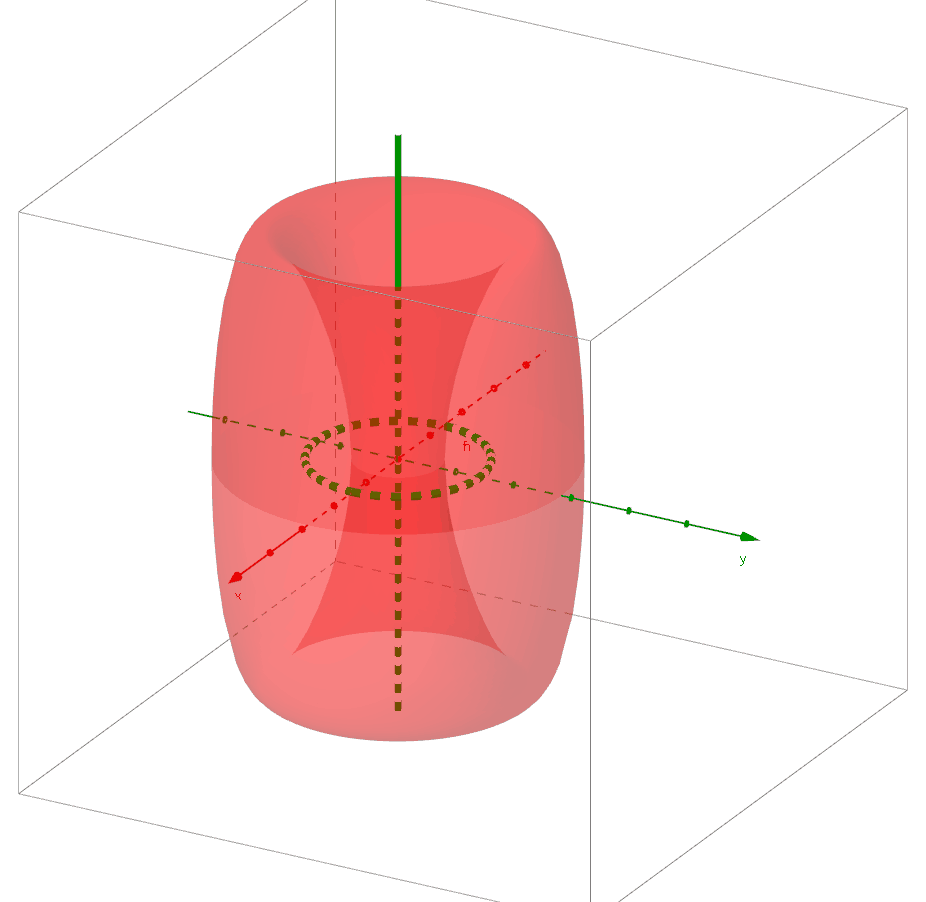}
 \caption{The invariant torus $T_{4/5}$ for an elliptic flow in Siegel's model.}\label{tore_inv}
\end{figure} 
 
 Let's now have a look at the orbits of the flow  $\phi_t^{\Log(E_{\alpha,\beta,\gamma})}$. Remark that the orbit of a point is included in a unique torus $T_r$, and that every orbit included in $T_r$ is the image of a fixed orbit by an element  $E_{\theta_1,\theta_2,-(\theta_1+\theta_2)}$. The torus $T_r$ is then foliated by these orbits. We fix $r \in ]0,1[$. We consider two cases
 
 \paragraph{Case 1 : $\frac{\alpha}{\beta} \notin \mathbb{Q}$.}
 
In this case, the angles of rotation in $T_r$ for $\phi_t$ are $(2\alpha + \beta)t$ and $(2 \beta + \alpha)t$. Since their quotient is irrational, an orbit is an injective immersion of a line and it is dense in $T_r$.

 \paragraph{Case 2 : $\frac{\alpha}{\beta} \in \mathbb{Q}$.} 
 
 In this case, the angles of rotation in $T_r$ for $\phi_t$ are $(2\alpha + \beta)t$ and $(2 \beta + \alpha)t$. Their quotient is rational; let's denote it $\frac{p}{q}$ in its irreducible form. The orbits are periodic and of slope $\frac{p}{q}$ in $T_r$ : they are torus knots of type $(p,q)$, knotted around $C_1$ and $C_2$. We can see an example in figure \ref{noeud_7_11}.

 \begin{figure}[H]
\center
 \includegraphics[width=5cm]{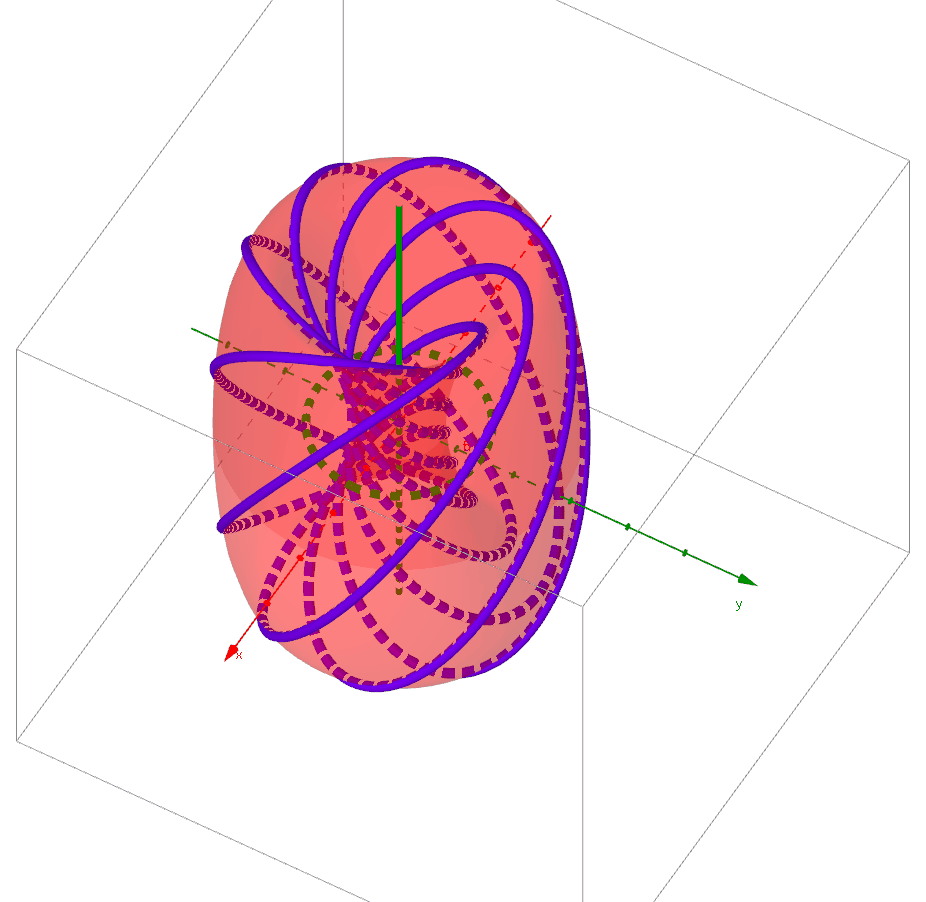}
 \caption{An orbit for $\frac{2\alpha + \beta}{2\beta + \alpha} = \frac{7}{11}$ : a torus knot of type (7,11).}\label{noeud_7_11}
\end{figure} 
 
\begin{rem} \label{rem_noeud_tor}
 If $p$ are $q$ are different from $\pm 1$, the orbit of a point of $T_r$ is a \emph{torus knot of type} $(p,q)$ and is knotted in $\dh2c$. If $p$ or $q$ equals $\pm 1$, then the orbit is not knotted; it is isotopic to an unknotted circle of $\dh2c$. We can see an example in figure \ref{noeud_1_3}.
\end{rem}
 
\begin{figure}[H]
\center
 \includegraphics[width=5cm]{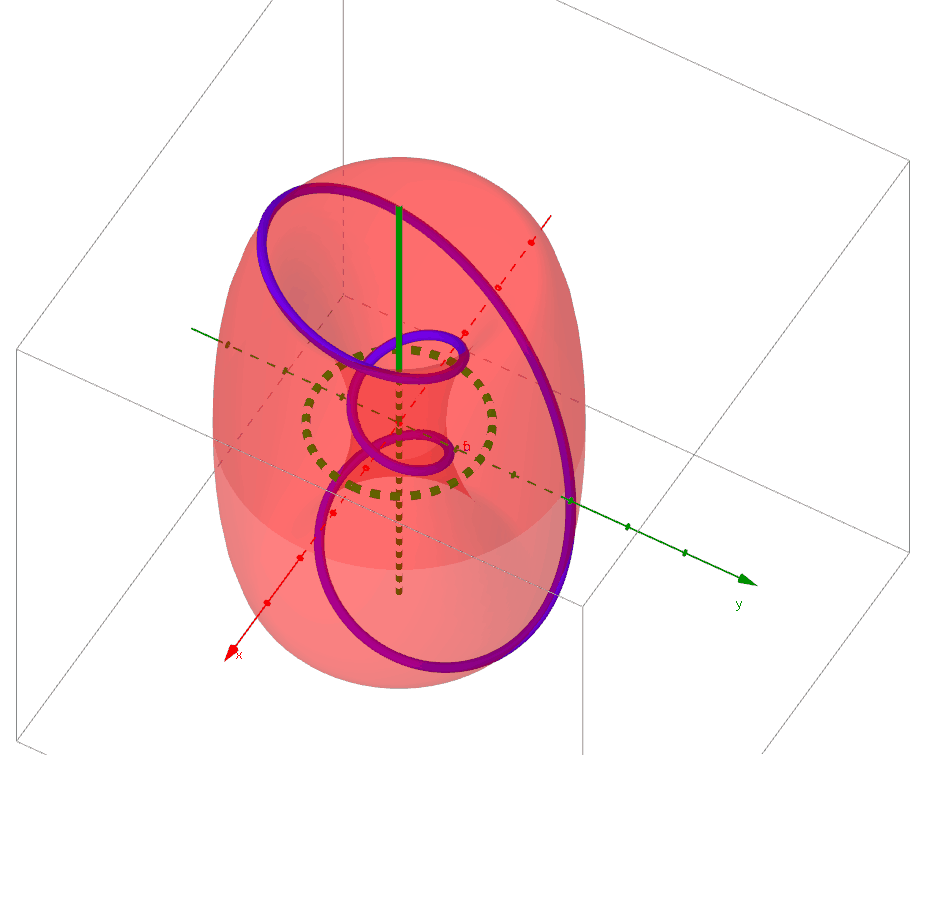}
 \caption{An orbit for $\frac{2\alpha + \beta}{2\beta + \alpha} = \frac{1}{3}$ : it is unknotted.}\label{noeud_1_3}
\end{figure}  
 
 \begin{defn}
  Let $n,p$ and $q$ be relatively prime integers with $|p| \geq |q|$. We say that an elliptic element $U \in \pu21$ is \emph{of type $(\frac{p}{n},\frac{q}{n})$} if $U$ is conjugate to $E_{\alpha,\beta,\gamma}$ with $\alpha = \frac{2p-q}{3n}$, $\beta = \frac{2q-p}{3n}$ and $\gamma = - \alpha - \beta = \frac{-p-q}{3n}$. In this case, $\frac{2\alpha + \beta}{2\beta + \alpha} = \frac{p}{q}$ and the orbits of the flow $\phi_t^{\Log(U)}$ are its two invariant $\mathbb{C}$-circles and torus knots of type $(p,q)$.
 \end{defn} 
 
\begin{rem}
 \begin{enumerate}
  \item Only some elliptic elements are of some type $(\frac{p}{n},\frac{q}{n})$. We will see later that elements of some type $(\frac{p}{n},\frac{q}{n})$ are the ones for which our construction happens to work.
  \item The trace of an elliptic element gives its three eigenvalues, but it is not enough to determine the type of the element. Indeed, an element of the same trace as an elliptic of type $(\frac{p}{n},\frac{q}{n})$ will have the same eigenvalues but not necessarily the same eigenvalue associated to its fixed point in $\h2c$. Thus, elements of type $(\frac{p}{n},\frac{q}{n})$, $(\frac{-p}{n},\frac{q-p}{n})$ and $(\frac{p-q}{n},\frac{-q}{n})$ have the same trace but are not conjugate.
 \end{enumerate}
 \end{rem}

\subsubsection{Loxodromic flows}

 Consider a loxodromic element in $\su21$ in Siegel's model. Perhaps after a conjugation, we can suppose that it is

 \[T_{\lambda} = \begin{pmatrix}
\lambda & 0 & 0 \\
0 & \frac{\overline{\lambda}}{\lambda} & 0 \\
0 & 0 &  \frac{1}{\overline{\lambda}} \\
\end{pmatrix} \]
where $\lambda \in \mathbb{C}$ is of modulus > 1. We have then $\lambda = re^{i\alpha}$, 	with $\alpha \in \mathbb{R}$ and $r>1$. We suppose that $\alpha$ is small enough, so the series $\Log(T_{\lambda})$ converges. In this case we have 
\[\Log(T_{\lambda}) = \begin{pmatrix}
\log(r) + i\alpha & 0 & 0 \\
0 & -2i\alpha & 0 \\
0 & 0 &  -\log(r) +i \alpha \\
\end{pmatrix}\]

The flow of the associated vector field acts on $\dh2c$ by:
\[\phi_t^{\Log(T_{\lambda})} \begin{bmatrix}
-\frac{1}{2}(|z|^2+is) \\ z \\ 1
\end{bmatrix}
=
\begin{bmatrix}
-\frac{1}{2}(r^{2t}|z|^2+ir^{2t}s) \\ r^te^{-3it\alpha}z \\ 1
\end{bmatrix} \]

In coordinates $(z,s)\in \mathbb{C} \times \mathbb{R}$ the action is given by $(z,t) \mapsto (\mu_t z , |\mu_t|^2s)$ where $\mu_t = r^{t}e^{-3i\alpha t}$.

\begin{rem}
 The flow $\phi_t$ fixes globally the points $\begin{bmatrix} 0 \\ 0 \\1 \end{bmatrix}$ and $\begin{bmatrix} 1 \\ 0 \\ 0 \end{bmatrix}$ and stabilizes the $\mathbb{C}$-circle joining them, which is called the \emph{axis} of $[T_{\lambda}]$. It also stabilizes the arcs $C_+$ and $C_-$ of this $\mathbb{C}$-circle, given by \[C_+ =  \left\{ \begin{bmatrix}
-\frac{1}{2}is \\ 0 \\ 1
\end{bmatrix} \in \dh2c \mid s >0  \right\} \text{ and } C_- =  \left\{ \begin{bmatrix}
-\frac{1}{2}is \\ 0 \\ 1
\end{bmatrix} \in \dh2c \mid s < 0  \right\}.\] Furthermore, for all $u\in \dh2c$, we have \[\lim_{t \to + \infty}\phi_t(u)= \begin{bmatrix} 1 \\ 0 \\ 0 \end{bmatrix} \text{ and } \lim_{t \to - \infty}\phi_t(u)= \begin{bmatrix} 0 \\ 0 \\ 1 \end{bmatrix}.\]

\end{rem}

\begin{figure}[H]
\center
 \includegraphics[width=5cm]{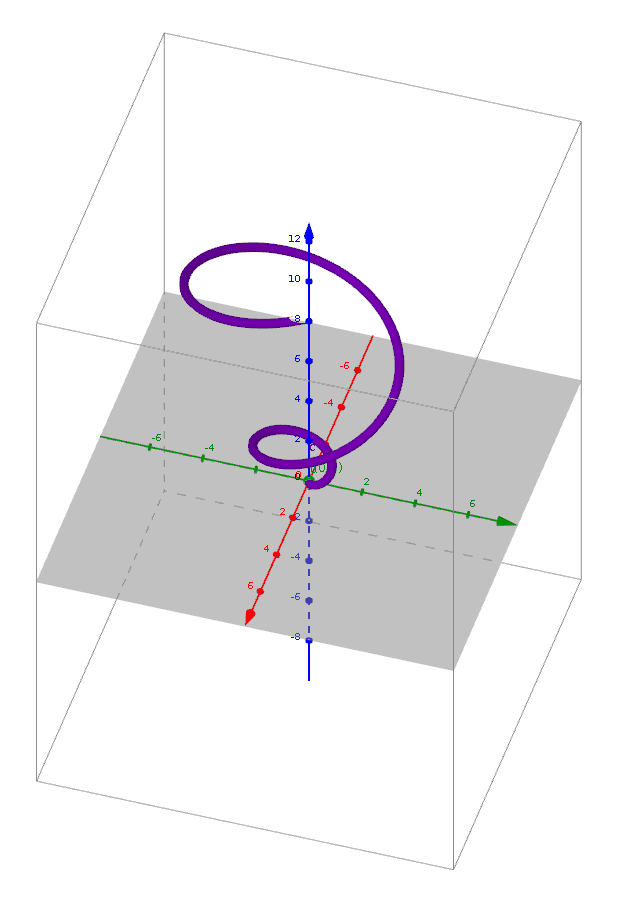}
 \caption{An orbit of a loxodromic flow in Siegel's model.}\label{flot_lox}
\end{figure} 

In the same way as in the elliptic case, we have flow-invariant objects, related to the centralizer of $T_\lambda$. We make the following remark:

\begin{rem}
 The centralizer of $T_{ \lambda}$ is $C(T_{\lambda})= \left\{T_{\mu} \mid \mu \in \mathbb{C}^* \right\}$. The orbits of this subgroup in $\dh2c$ are the two fixed points of $T_\lambda$, $C_+$, $C_-$ and the subsets $P_r$ for $r\in \mathbb{R}$, defined by
 
 \[P_r = \left\{ \begin{bmatrix}
-\frac{1}{2}(|z|^2+is) \\ z \\ 1
\end{bmatrix} \in \dh2c \mid \frac{s}{|z|^2}=r  \right\}.\] 

 The orbits $P_r$ are punctured paraboloids (with missing point $(0,0)$). Topologically, they are embedded cylinders in $\dh2c \setminus \left\{ \begin{bmatrix} 1 \\ 0 \\ 0 \end{bmatrix}, \begin{bmatrix} 0 \\ 0 \\ 1 \end{bmatrix} \right\}$, all invariant by the action of $\phi_t$. We can see an example in figure \ref{cyl_inv_lox}.
\end{rem}

\begin{figure}[H]
\center
 \includegraphics[width=5cm]{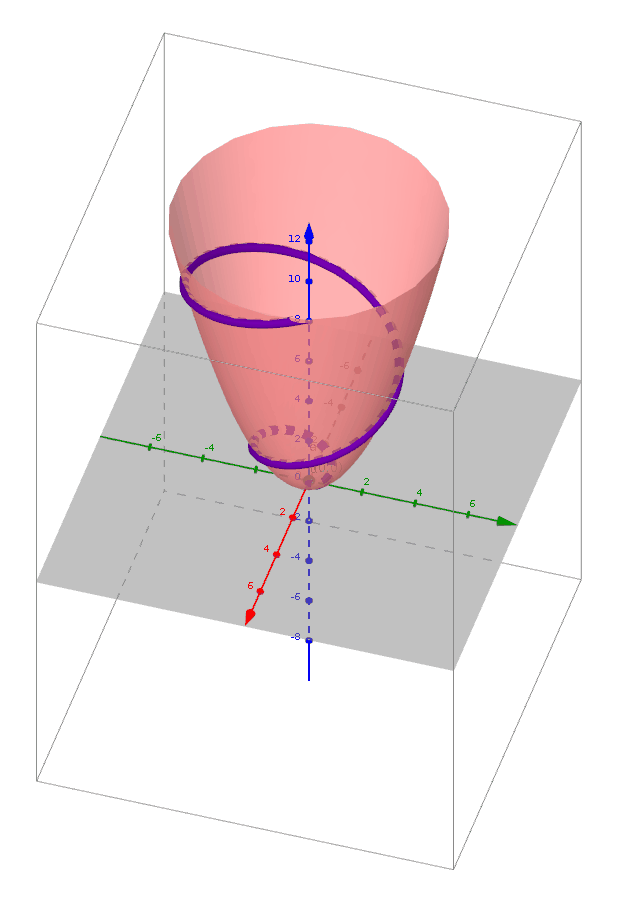}
 \caption{A cylinder invariant under a loxodromic flow in Siegel's model.}\label{cyl_inv_lox}
\end{figure}

\subsubsection{Unipotent flows}
Consider now a unipotent element of $\su21$ in Siegel's model. Perhaps after a conjugation, we can suppose that it is, for $(z,s) \in \mathbb{C} \times \mathbb{R}$

\[P_{z,s} = \begin{pmatrix}
1 & -\overline{z} & -\frac{1}{2}(|z|^2 + is) \\
0 & 1 & z \\
0 & 0 & 1 \\
\end{pmatrix} . \]

The series $\Log(P_{z,s})$ converges and we have 
\[\Log(P_{z,s}) = \begin{pmatrix}
0 & -\con{z} & -\frac{is}{2} \\
0 & 0 & z \\
0 & 0 & 0 \\
\end{pmatrix}\]

The flow of the associated vector field acts on $\dh2c$ by:
\[\phi_t^{\Log(P_{z,s})} \begin{bmatrix}
-\frac{1}{2}(|z'|^2+is') \\ z' \\ 1
\end{bmatrix}
=
\begin{bmatrix}
-\frac{1}{2}(|z'+tz|^2+i(s'+ts-2t\Im(\con{z}z'))) \\ z'+tz \\ 1
\end{bmatrix} \]

In coordinates $(z,s)\in \mathbb{C} \times \mathbb{R}$ the action is given by $(z',s') \mapsto (z'+tz,s'+ts-2t\Im(\con{z}z'))$.

In this coordinates, the orbits of the flow are straight lines, like those in figure \ref{flot_parab}.

\begin{figure}[H]
\center
 \includegraphics[width=5cm]{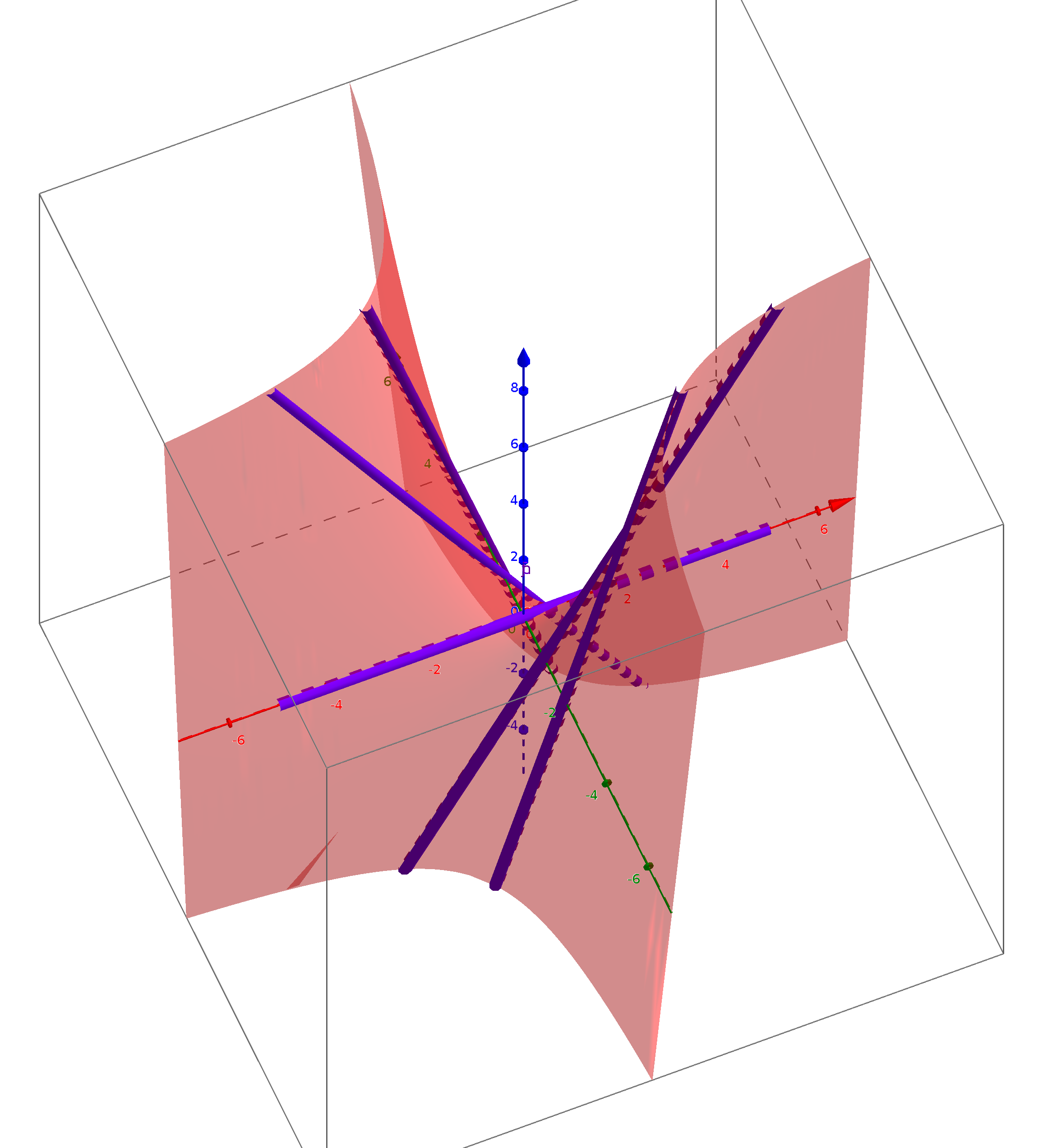}
 \caption{Some orbits of a unipotent flow in Siegel's model.}\label{flot_parab}
\end{figure}

\begin{rem}
 If $z=0$, then $[P_{z,s}]$ is called a \emph{vertical parabolic} element and all orbits of the flow are vertical lines. If not, $[P_{z,s}]$ is called a \emph{horizontal parabolic} and the orbits of the flow are lines with different slopes.
\end{rem}

\begin{rem}
 If $z \neq 0$, the centralizer of $P_{(z,s)}$ is $C(P_{(z,s)})= \left\{P_{(rz,s')} \mid (r,s') \in \mathbb{R}^2 \right\}$. The orbits of this subgroup in $\dh2c$ are the fixed point of $P_{(z,s)}$ and the subsets $S_r$ for $r\in \mathbb{R}$, defined by
 
 \[S_r = \left\{ \begin{bmatrix}
-\frac{1}{2}(|w|^2+is') \\ w \\ 1
\end{bmatrix} \in \dh2c \mid \Im(\frac{w}{z})=r  \right\}.\] 

The orbits $S_r$ are vertical planes in Siegel's model, all invariant under the action of $\phi_t$. We can see an example in figure \ref{nappe_inv_parab}.
\end{rem}

\begin{figure}[H]
\center
 \includegraphics[width=5cm]{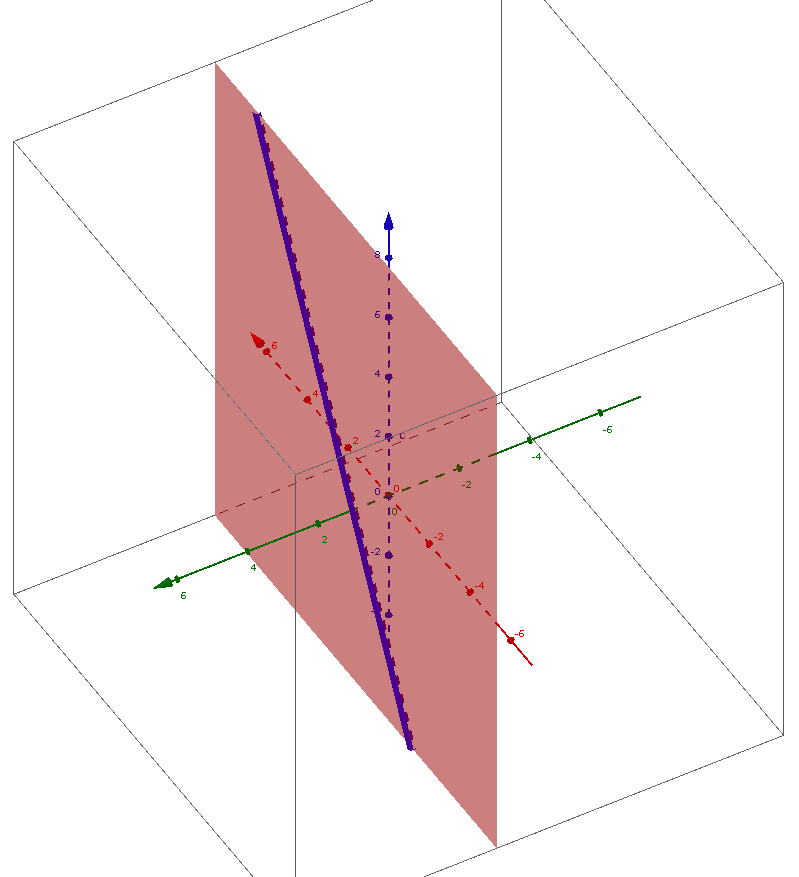}
 \caption{An invariant plane under a horizontal unipotent flow in Siegel's model.}\label{nappe_inv_parab}
\end{figure}

\subsubsection{Ellipto-parabolic flows}
 Finally, we consider nonregular ellipic elements, also called ellipto-parabolic. In Siegel's model, Perhaps after a conjugation, we can consider the element:
 
  \[E = e^{i\theta}\begin{pmatrix}
1 & 0 & -\frac{i}{2} \\
0 & e^{-3i\theta} & 0 \\
0 & 0 &  1 \\
\end{pmatrix}. \]

If $\theta$ is small enough, then the series $\Log(E)$ converges and we have  

 \[\Log(E) = \begin{pmatrix}
i\theta & 0 & -\frac{i}{2} \\
0 & -2i\theta & 0 \\
0 & 0 &  i\theta \\
\end{pmatrix}. \]

In coordinates $(z,s) \in \mathbb{C} \times \mathbb{R}$, the flow of the associated vector field acts by $(z,s) \mapsto (e^{-3i\theta t}z, s+t)$.

The orbits of the flow are the $\mathbb{C}$-circle invariant by $[E]$ and spirals turning around it, as in figure \ref{flot_ep_1}

\begin{figure}[H]
\center
 \includegraphics[width=5cm]{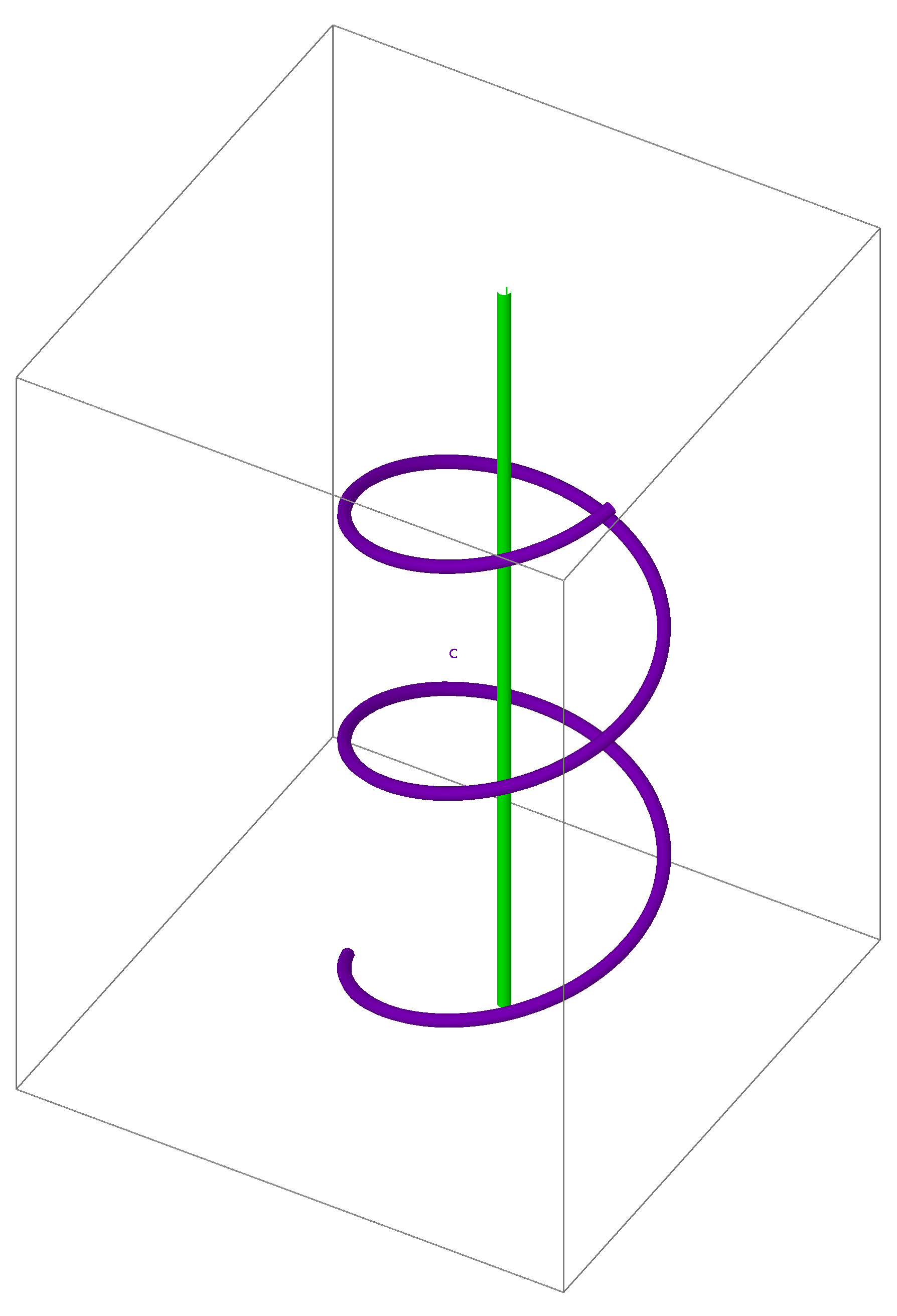}
 \caption{An orbit of an ellipto-parabolic flow in Siegel's model.}\label{flot_ep_1}
\end{figure} 

\begin{rem}
 The centralizer $C(E)$ of $E$ in $\su21$ is the set of elements of the form: 
 \[e^{i\varphi}\begin{pmatrix}
1 & 0 & -\frac{it}{2} \\
0 & e^{-3i\varphi} & 0 \\
0 & 0 &  1 \\
\end{pmatrix} \]
for $(t,\varphi) \in \mathbb{R}^2$. The orbits of this subgroup are the $\mathbb{C}$-circle invariant by $[E]$ and cylinders around it, as in figure \ref{nappe_inv_ep} 
\end{rem}

\begin{figure}[H]
\center
 \includegraphics[width=5cm]{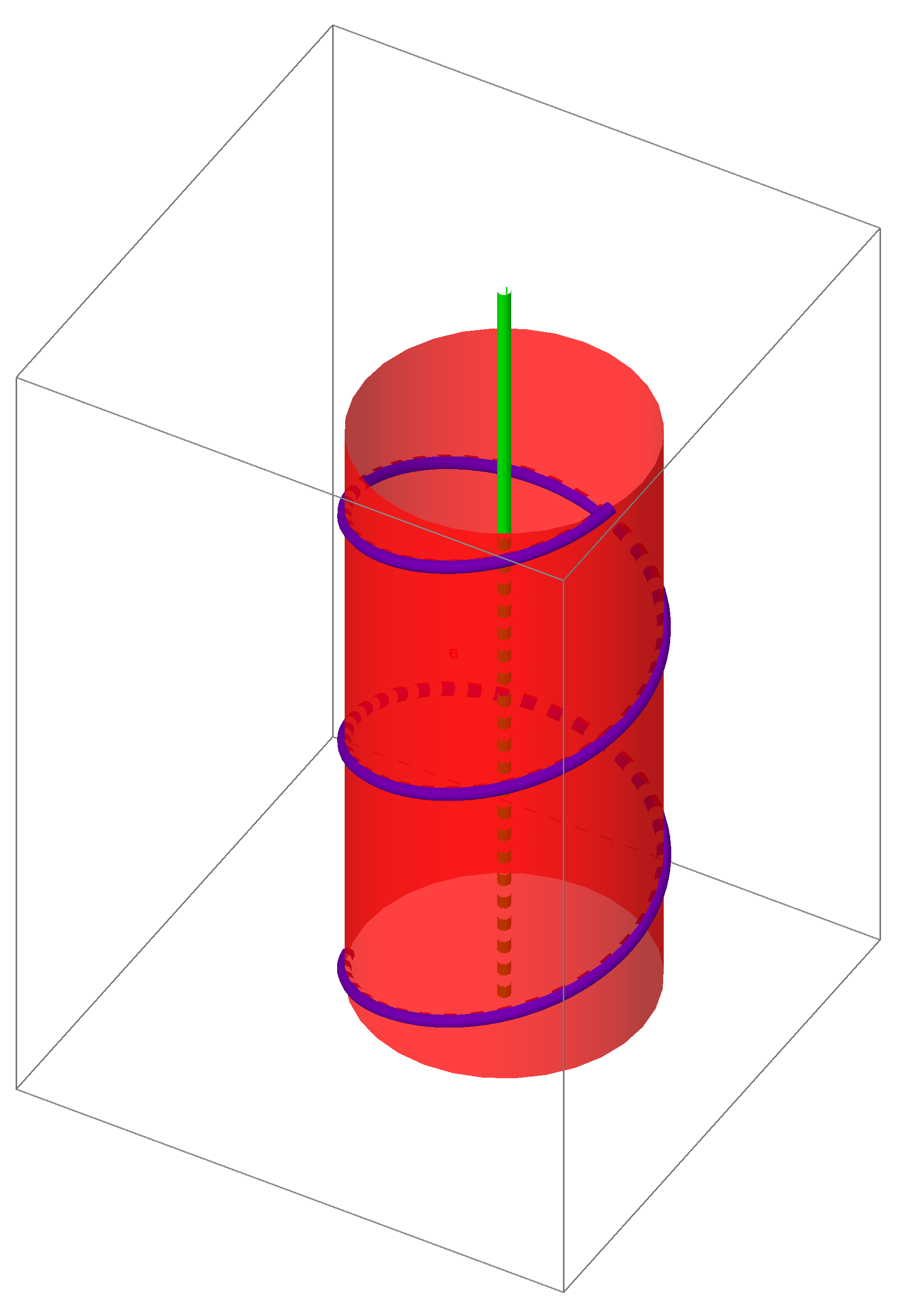}
 \caption{A surface invariant under the flow of an ellipto-parabolic element in Siegel's model.}\label{nappe_inv_ep}
\end{figure}

\subsection{Some remarks on the convergence of regular elements}

  The projection $\su21 \rightarrow \pu21$ is a covering of order 3 ; in order to study the convergence in $\pu21$ we can focus on the convergence in $\su21$.
 
 Let $(U_n)_{n \in \mathbb{N}}$ be a sequence of regular elements of $\su21$ converging to $U \in \su21 \setminus \mathbb{C}\mathrm{Id}$.
 
  If $U$ is regular, then the convergence is given by the convergence of eigenvectors and eigenvalues. We consider now the case where $U$ is not regular. The continuity of eigenvectors and eigenvalues gives the following lemma.

 \begin{lemme}
   Let $(U_n)_{n \in \mathbb{N}}$ be a sequence of regular elements of $\su21$ converging to $U \in \su21 \setminus \mathbb{C}\mathrm{Id}$. Let $(([u_n],\alpha_n),([v_n],\beta_n),([w_n],\gamma_n))$ be the eigenvectors and eigenvalues of $U_n$ in some order. Then, perhaps after changing the labeling, $(([u_n],\alpha_n),([v_n],\beta_n),([w_n],\gamma_n))$ converges to $(([u],\alpha),([v],\beta),([w],\gamma)) $ in $(\mathbb{CP}^2 \times \mathbb{C})^3$, where $([u],\alpha),([v],\beta),([w],\gamma)$ are (possibly equal) eigenvectors and eigenvalues of $U$. 
 \end{lemme} 
 
 Consider the case when $U$ is horizontal parabolic. Then, $U$ has a unique fixed point $[p]\in \cp2$, which is in $\dh2c$, and its eigenvalues can be chosen all equal to 1.
 Using the above lemma, we deduce that  $(\alpha_n,\beta_n,\gamma_n) \rightarrow (1,1,1)$ and $([u_n],[v_n],[w_n]) \rightarrow ([p],[p],[p])$. From a geometric point of view on $\h2c \cup \dh2c$ we make the two following remarks:
 
 \begin{rem}\label{rem_axes_lox_loin}
  If the $U_n$ are loxodromic of axes $l_n$ then the $l_n$ leave every compact of $\h2c \cup \dh2c \setminus \{[p]\}$.
 \end{rem}
 
 \begin{rem}\label{rem_axes_ell_loin}
  If the $U_n$ are elliptic, they have each two invariant complex geodesics $l_n^{(1)}$ and $l_n^{(2)}$ (the polar lines $[v_n]^{\perp}$ and $[w_n]^{\perp}$ if $[u_n]$ is the fixed point of $U_n$ in $\h2c$). Then the $l_n^{(i)}$ leave every compact of $\h2c \cup \dh2c \setminus \{[p]\}$.
 \end{rem}
 
 These two remarks will be crucial when understanding the geometry of deformations of \CR {} structures by considering a developing map.

\section{Regular surgeries}\label{section_surgeries}

\subsection{The Ehresmann-Thurston principle}
 We are going to study \CR {} structures on a 3-manifold $M$. Let's begin by recalling the formalism of $(G,X)$-structures, that will give us the language to use. In the definitions, $X$ will be a smooth connected manifold and $G$ a subgroup of the diffeomorphisms of $X$ acting transitively and analytically on $X$. We will focus on the case where $X = \dh2c$ and $G= \pu21$.

\begin{defn}
 A \emph{$(G,X)$-structure} on a manifold $M$ is a pair $(\Dev, \rho)$, up to isotopy, of a local diffeomorphism $\Dev : \widetilde{M} \rightarrow X$ and a group homomorphism $\rho : \pi_1(M) \rightarrow G$ such that for all $x \in \widetilde{M}$ and all $g \in \pi_1(M)$ we have $\Dev(g\cdot x)=\rho(g)\cdot \Dev(x)$ for the group actions of $\pi_1(M)$ on $\widetilde{M}$ and of $G$ on $X$.
 
 We say that $\Dev$ is the \emph{developing map} of the structure and $\rho$ its \emph{holonomy}.
\end{defn}

\begin{rem}
 We identify two structures if they are $G$-equivalent, i.e. if there is a $g \in G$ such that the developing maps $\Dev_1$ and $\Dev_2$ satisfy $\Dev_2=g\circ \Dev_1$. In this case, the holonomy representations are conjugate and satisfy $\rho_2 = g \rho_1 g^{-1}$.
\end{rem}

 The definition we just gave is not the usual one. We make then the following remark :

\begin{rem}
 This definition is equivalent to the usual definition of a $(G,X)$-structure as an atlas of charts of $M$ taking values in $X$ and whose transition maps are given by elements of $G$. A couple $(\Dev,\rho)$ immediately gives such an atlas, but the construction of $(\Dev,\rho)$ from an atlas requires more work. See for example Thurston's notes \cite{gt3m}. Nevertheless, we will use both definitions: the first in order to deform a structure, and the second to construct a new one.

\end{rem}

 We will use sometimes manifolds with boundary, but the definition of $(G,X)$-structure easily extends to this case. From now on, we consider a compact 3 dimensional manifold $M$ with (possibly many) torus boundaries. We are going to study \CR {} structures on $M$, where the model space $X$ is $\dh2c$ and the group $G$ is $\pu21$.

\begin{defn}
 A \emph{\CR {} structure} is a $(\pu21,\dh2c)$-structure.
\end{defn}

 In order to deform the structure using the Ehresmann-Thurston principle that we state below, the essential object are the representations of $\pi_1(M)$ taking values in $\pu21$.

\begin{notat}
Let $\mathcal{R}(\pi_1(M),G)$ be the set of representations of $\pi_1(M)$ taking values in $G$, endowed with the topology of pointwise convergence.
\end{notat}

 We are going to work with deformations of some structures. In order to state the results on a deformation, we will need to be "far enough from the boundary" or "close to the boundary". We are going to consider the union of $M$ with a thickening of its boundary to be able to state the results precisely.

\begin{notat}
If $s \in \mathbb{R}^+$, denote $M_{[0,s[}$ the union of $M$ with a thickening of its boundary. Thus, $M_{[0,s[} =( M \cup (\partial M \times [0,s[) )/ \sim$ where we identify $\partial M $ to $\partial M \times \{0\}$. We consider those manifolds as included into eachother, in such a way that if $s_1\leq s_2$, then $M_{[0,s_1[} \subset M_{[0,s_2[}$
\end{notat}

\begin{rem}
 The manifolds $M_{[0,s[}$ are all homeomorphic to the interior of $M$. We use these cuts in order to get a suitable convergence "far enough" from the boundary of  $M$ for geometric structures.
\end{rem}

Let's state the Ehresmann-Thurston principle, which says that we only need to deform in $\mathcal{R}(\pi_1(M),G)$ the holonomy of a $(G,X)$-structure to have a deformation of the structure itself. A proof can be found in the paper of Bergeron and Gelander \cite{bergeron_gelander} or in the survey of Goldman \cite{goldman_manifolds}.

\begin{thm}[Ehresmann-Thurston principle]
 Suppose $M_{[0,1[}$ has a $(G,X)$-structure of holonomy $\rho_0$. For all $\epsilon > 0$, if $\rho \in \mathcal{R}(\pi_1(M),G)$ is a deformation close enough from $\rho_0$, then there is a $(G,X)$-structure on $M_{[0,1-\epsilon[}$ of holonomy $\rho$ and close to the first structure on $M_{[0,1-\epsilon[}$ in the $\mathcal{C}^1$ topology.
\end{thm}


\subsection{Surgeries}
 As in the real hyperbolic case, we consider Dehn surgeries of $M$, which are, from a topological point of view, a gluing of solid tori on the torus boundaries of $M$. We attempt to extend a \CR {} structure on $M$ to one of its surgeries. We show a result very similar to the one showed by Schwartz in \cite{schwartz}, but with some differences. On the one hand, our hypotheses are weaker than Schwartz's and we obtain a geometric structure. On the other hand we do not know if the structure is obtained as a quotient of an open set of $\dh2c$ by the action of a subgroup of $\pu21$.

\subsubsection{Thickenings and lifts}

Let's begin by fixing notation for a torus boundary component, one of its lifts and the associated peripheral holonomy. We denote $\widetilde{M}$ the universal cover of $M$ and $\pi : \widetilde{M} \rightarrow M$ its projection. We state all results for a single torus boundary component in order to avoid heavy notation, but the same statements work for several boundary components.

\begin{notat}
Let $T$ be a fixed torus boundary of $M$. For $s\in [0,1[$, let $T_s = T \times \{s\} \subset M_{[0,1[}$, and, for an interval $I \subset [0,1[$, let $T_I= \cup_{s\in I} T_s = T \times I \subset M_{[0,1[}$. Let $\widetilde{T}_{[0,1[}$  denote some connected component of $\pi^{-1}(T_{[0,1[}) \subset \widetilde{M}_{[0,1[}$ : it is a universal cover of $T_{[0,1[}$ embedded in $\widetilde{M}_{[0,1[}$. Finally, for $s \in [0,1[$, set $\widetilde{T}_s = \pi^{-1}(T_s) \cap \widetilde{T}_{[0,1[}$ and, for an interval $I \subset [0,1[$, $\widetilde{T}_I= \cup_{s\in I} \widetilde{T}_s$.

\end{notat}

We make some remarks on the choices made by taking these notations:

\begin{rem}
For all $s\in [0,1[$, $\widetilde{T}_s$ is homeomorphic to $\mathbb{R}^2$. Furthermore, $\widetilde{T}_I$ is homeomorphic to $\mathbb{R}^2 \times I$.  
\end{rem}

\begin{rem}
 The choice of $\widetilde{T}_{[0,1[}$ fixes an injection of the fundamental group of $T$ into the fundamental group of $M$, by identifying $\pi_1(T)$ to the stabilizer of $\widetilde{T}_{[0,1[}$ for the action of $\pi_1(M)$ on $\widetilde{M}_{[0,1[}$.
\end{rem}

\begin{notat}
 With the fixed injection of $\pi_1(T)$ into $\pi_1(M)$, by restricting the holonomy $\rho$ of a $(G,X)$-structure we have a \emph{peripheral holonomy} $h_{\rho} : \pi_1(T) \rightarrow G$.
\end{notat}

\begin{notat}
 We denote by $\mathcal{R}_1(\pi_1(M),G) \subset \mathcal{R}(\pi_1(M),G)$ the set of representations $\rho$ such that $h_\rho$ is generated by a single element. When $\rho \in \mathcal{R}_1(\pi_1(M),\pu21)$ has $[U]\in \pu21$ as a preferred generator for its image, we write $\phi_t^\rho$ for $\phi_t^{\Log([U])}$.
\end{notat}

 \subsubsection{Horotubes}

 We are going to use the definitions related to horotubes given by Schwartz in \cite{schwartz}:
 
\begin{defn}
 Let $[P] \in \pu21$ be a parabolic element of fixed point $p\in \dh2c$. A \emph{$[P]$-horotube} is an open set $H$ of $\dh2c \setminus \{p\}$, invariant under $[P]$ and such that the complement of $H/\langle [P] \rangle$ in $(\dh2c \setminus \{p\})/\langle [P] \rangle$ is compact. 
\end{defn} 
 
 In order to work with more regular objects, we often ask horotubes to be \emph{nice}:
 
 \begin{defn}
  A $[P]$-horotube $H$ is \emph{nice} if $\partial H$ is a smooth cylinder invariant by the flow $\phi_t^{\Log([P])}$.
 \end{defn}

 \begin{rem}
  If $H$ is a nice $[P]$-horotube, then $\partial H$ is the orbit for $\phi_t^{\Log([P])}$ of an embedded circle of $\dh2c \setminus \{p\}$. We can see an example in figure \ref{horotube56}.
 \end{rem}
 
\begin{figure}[H]
\center
 \includegraphics[width=10cm]{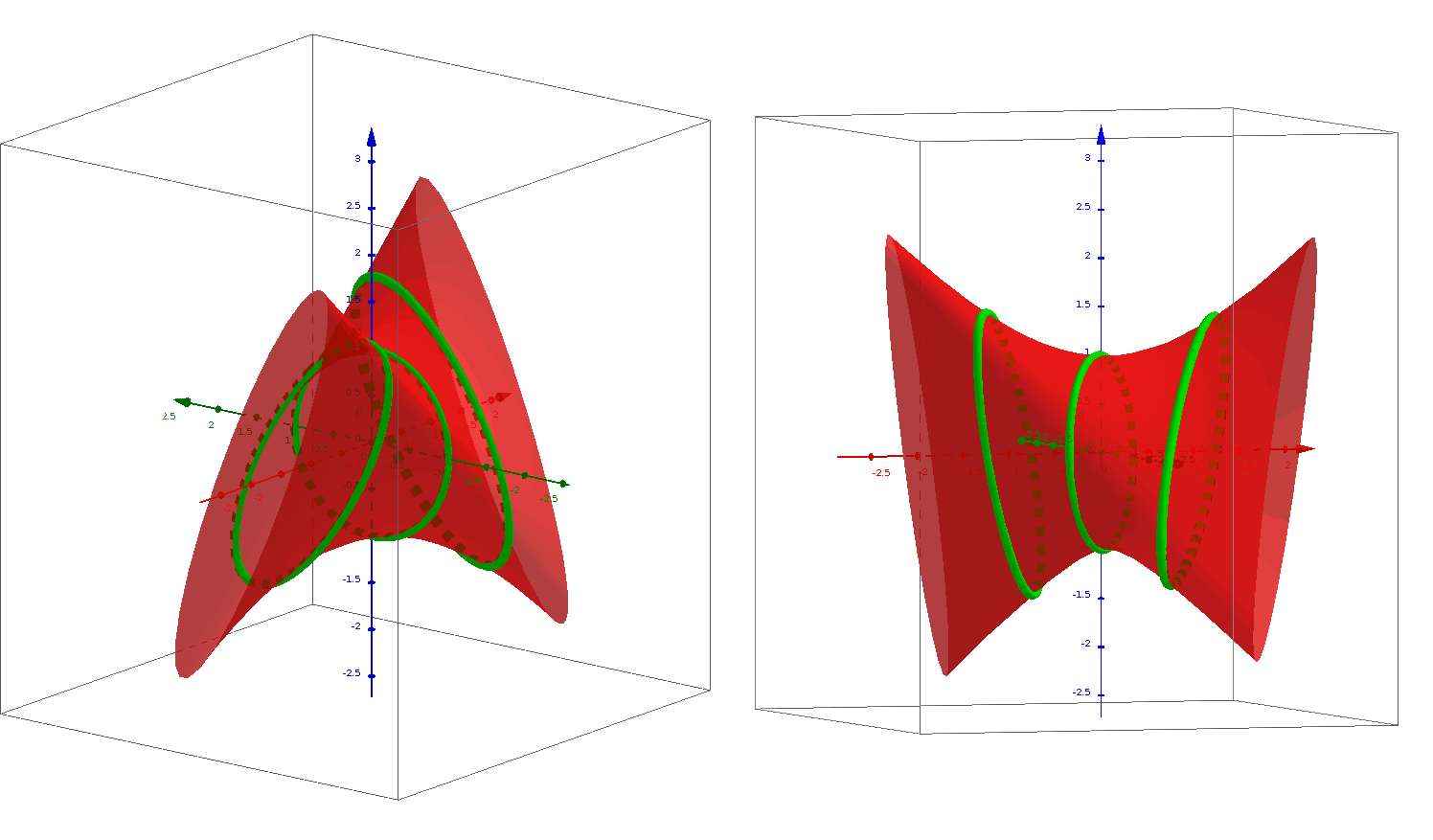}
 \caption{The boundary of a nice horotube in Siegel's model. The horotube is outside the red surface.}\label{horotube56}
\end{figure} 

 The following remark says that, shrinking the horotube if necessary, we may assume it is nice.
 
 \begin{lemme}[2.7 of \cite{schwartz}] \label{rq_horotube_gentil}
  Let $H$ be a $[P]$-horotube. Then, there is a nice $[P]$-horotube $H'$ such that $H' \subset H$ and $(H \setminus H')/\langle [P] \rangle$ is of compact closure in $(\dh2c \setminus \{p\}) / \langle [P] \rangle$.
 \end{lemme}

From now, we suppose that $M_{[0,1[}$ has a \CR { }  structure of developing map $\Dev_0$ and holonomy $\rho_0$. We also make two more hypotheses:

\begin{enumerate}
  \item The image of the peripheral holonomy $h_{\rho_0}$ is unipotent of rank 1 and generated by an element $[U_0] \in \pu21$.
  \item There is $s \in [0,1[$ such that $\Dev_0(\widetilde{T}_{[s,1[})$ is a $[U_0]$-horotube.
\end{enumerate}

\subsubsection{Marking of $\pi_1(T)$} \label{sect_marquage}
 We are going to fix a marking of $\pi_1(T)$ naturally deduced from the structure given by $\Dev_0$ and $\rho_0$. This marking will be useful to identify the Dehn surgeries obtained when deforming the structure. It is essentially the same marking as the one given in chapter 8 of \cite{schwartz};  its definition uses the two hypotheses given above.
 
 \begin{notat}
  Fix $s'\in [s,1[$ and $x_0 \in \Dev_0(T_{s'})$.
  Let $l$ be the loop given by the projection of $t \mapsto \phi_t^{\rho_0}(x_0)$. As $h_{\rho_0}(l)=[U_0]$ generates the image of $h_{\rho_0}$ and since a unipotent subgroup of $\pu21$ has no torsion, $l$ is a primitive element of $\pi_1(T)$. 
 \end{notat} 
 
 \begin{notat}
  Since $h_{\rho_0}$ is unipotent of rank one and a unipotent subgroup of $\pu21$ has no torsion, its kernel is generated by a primitive element $m$. We orient $m$ in such a way that $(l,m)$ is a direct basis of $\pi_1(T)$ (for the orientation given by the inside normal in the horotube).
 \end{notat}

 \begin{rem}
  The definition of $l$ and $m$ does not depend on the choice of $s'$ nor of $x_0$. Nevertheless, we made a choice for orientations. The one for $m$ is explicit, but the orientation of $l$ is given by the choice of $[U_0]$ or $[U_0]^{-1}$ as a generator of the image of $h_{\rho_0}$.
 \end{rem}

 \begin{rem}
  In \cite{schwartz}, Schwartz gives a "canonical" choice for the orientations of $l$ and $m$ (denoted $\beta$ and $\alpha$). It is almost the same choice as the one made above, but he has a preferred choice for $[U_0]$. Note that the marking $(l,m)$ given here might not be the usual one. If we have another marking of $\pi_1(T)$, for example when $M$ is a knot complement, changing markings can be done easily when $\rho_0$ is known explicitly.
 \end{rem}

\begin{figure}[H]
\center
 \includegraphics[width=5cm]{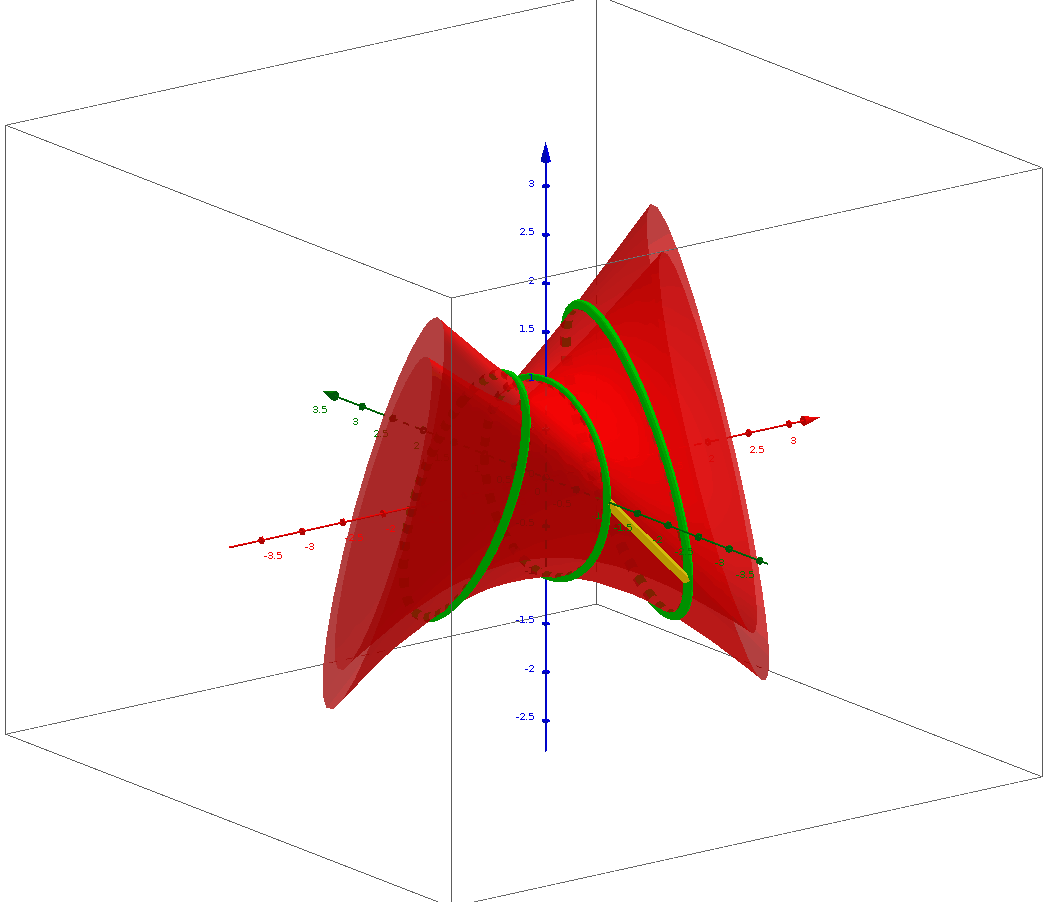}
 \caption{The curve $m$ (in green) and the curve $l$ (in yellow) in the image of $\Dev_0(\widetilde{T}_{s'})$.}\label{horotube_marque}
\end{figure} 

\begin{defn}
 For two relatively prime integers $p,q$, we denote by $M^{(p,q)}$ the manifold obtained by gluing a solid torus $D^2 \times S^1$ on the boundary $T$ of $M$ such that the loop $pl+qm$ of $T$ becomes trivial in $D^2 \times S^1$. We refer to it as the \emph{Dehn surgery of $M$ of type $(p,q)$} or of \emph{slope $\frac{p}{q}$}.
\end{defn}

 In the real hyperbolic case, deforming the complete hyperbolic structure on $M$ gives structures on all but a finite number of Dehn surgeries $M^{(p,q)}$ of $M$, as it is shown in Thurston's notes \cite{gt3m}. The main idea to prove it is to deform the structure "far" from the cusp, cutting by $T$, look at the developing map near the boundary $T$ and then notice that a solid torus can be glued to this boundary. What follows, stated in the \CR {} case, is inspired by this technique. The deformation "far" from the cusp rises a developing map near $T$, and the manifolds that can be glued are solid tori only in some cases.

\subsubsection{A surgery theorem}
 We are now able to state a \CR {} surgery theorem. It says that in a neighborhood of the structure  $(\Dev_0,\rho_0)$, under some discreteness conditions, \CR {} structures on $M$ come from structures on Dehn surgeries of $M$, and in some cases another kind of surgery.

\begin{thm}\label{thm_chirurgie}
 Let $M$ be a three dimensional compact manifold with torus boundary components. Let $T$ be one boundary torus of $M$. Suppose that there is a \CR {} structure  $(\Dev_0,\rho_0)$ on $M_{[0,1[}$ such that:
 \begin{enumerate}
  \item The peripheral holonomy $h_{\rho_0}$ corresponding to $T$ is unipotent of rank 1 and generated by an element $[U_0] \in \pu21$.
  \item There is $s \in [0,1[$ such that $\Dev_0(\widetilde{T}_{[s,1[})$ is a $[U_0]$-horotube.
\end{enumerate}  
 Then there is an open set $\Omega$ of $\mathcal{R}_1(\pi_1(M), \pu21)$ containing $\rho_0$ such that, for all $\rho \in \Omega$ for which the image of $h_\rho$ is generated by an element $[U] \in \pu21$, there is a  \CR {} structure  on $M$ with holonomy $\rho$. Furthermore, for the marking $(l,m)$ of $\pi_1(T)$ described above,
 
 \begin{enumerate}
  \item If $[U]$ is loxodromic, the structure extends on a  \CR {} structure  on the Dehn surgery of $M$ of slope $(0,1)$.
  \item If $[U]$ is elliptic of type $(\frac{p}{n},\frac{\pm 1}{n})$, the structure extends on a  \CR {} structure  on the Dehn surgery of $M$ of slope  $(n, \pm p)$.
  \item If $[U]$ is elliptic of type $(\frac{p}{n},\frac{q}{n})$ with $|p|,|q|>1$, the structure extends on a  \CR {} structure on the gluing of $M$ with a compact manifold with torus boundary $V(p,q,n)$. Furthermore $V(p,q,n)$ is a torus knot complement in the lens space $L(n,\alpha)$ where $\alpha \equiv p^{-1}q \mod n$.
 \end{enumerate}
\end{thm}

\begin{rem}
 The existence of the structure on $M$ is a consequence of the Ehresmann-Turston principle. To extend the structure we need a local surgery result, similar to the one given by Shwartz in \cite{schwartz}, and which is given in section \ref{section_proof}.
\end{rem}

\begin{rem}
 If $[U]$ is parabolic, then the theorem still holds, but the \CR {} structure extends to a thickening of $M$, which is homeomorphic to $M$ itself.
\end{rem}

\section{Deformations of the Deraux-Falbel structure on the figure eight knot complement} \label{section_deraux-falbel}

We are going to apply theorem \ref{thm_chirurgie} in the case of the \CR {} structure on the figure eight knot given by Deraux and Falbel in \cite{falbel}. We will use some results of \cite{deraux_uniformizations}, where Deraux describes a Ford domain for the structure, and also some results of \cite{character_sl3c}, where Falbel, Guilloux, Koseleff, Rouillier and Thistlethwaite describe the irreducible components of the $\mathrm{SL}_3(\mathbb{C})$ character variety of the figure eight knot complement.

\begin{notat} 
 In the rest of this section, we denote by $M$ the figure eight knot complement.
\end{notat}

\subsection{The Deraux-Falbel structure}
Let's begin by recalling quickly the results of Deraux and Falbel from \cite{falbel}. In that paper, the fundamental group of $M$ is given by $$\pi_1(M)=\langle g_1, g_2, g_3 \mid g_2 = [g_3,g_1^{-1}] , g_1g_2 = g_2 g_3 \rangle.$$
 
 The authors construct a uniformizable \CR {} structure on $M$ with unipotent peripheral holonomy. The holonomy representation $\rho_0$ is given by
 
 \[\rho_0(g_1)= [G_1] = \begin{bmatrix}
1 & 1 & -\frac{1}{2}-\frac{\sqrt{7}}{2}i \\
0 & 1 & -1 \\
0 & 0 &  1 \\
\end{bmatrix} , \hspace{0.5cm}
\rho_0(g_3)= [G_3] = \begin{bmatrix}
1 & 0 & 0 \\
-1 & 1 & 0 \\
-\frac{1}{2}+\frac{\sqrt{7}}{2}i & 1 &  1 \\
\end{bmatrix} \]  

\begin{rem}
 This representation is in the component $R_2$ of the character variety of \cite{character_sl3c}. For the notation of section 5.2 of \cite{character_sl3c} we have $A=g_3$ and $B=g_1$. With this notation, the usual longitude-meridian pair $(l_0,m_0)$ of the knot complements satisfies \[m_0 = g_3 \text{ and } l_0= g_1^{-1}g_3g_1g_3^{-2}g_1g_3g_1^{-1}.\]

 Furthermore, we check easily that $\rho_0(m_0)^3 = \rho_0(l_0)$, so $\rho_0(3m_0-l_0)= \mathrm{Id}$.
 \end{rem} 
 
 \begin{notat}
 From now on, in order to have the same notation as in Deraux's paper \cite{deraux_uniformizations}, we consider the pair $(l_1,m_1)$ obtained by conjugation by $g_2$, so that $m_1= g_2 g_3 g_2^{-1}= g_1$.
 
 Let $l=m_0$ and $m=3m_0-l_0$. In this way, $m$ generates $\ker (\rho_0)$ and $\rho(l)$ generates $\Im (\rho_0)$: this is a marking as the one in section \ref{sect_marquage}.
 \end{notat}

 \subsection{Checking the hypotheses}
 Recall the hypotheses of theorem \ref{thm_chirurgie}:
 
  \begin{enumerate}
  \item The peripheral holonomy $h_{\rho_0}$ is unipotent with image generated by an element $[U_0] \in \pu21$.
  \item There exists $s \in [0,1[$ such that $\Dev_0(\widetilde{T}_{[s,1[})$ is a $[U_0]$-horotube.
\end{enumerate} 
 
 The first hypothesis is satisfied by the Deraux-Falbel structure: indeed, the peripheral holonomy is unipotent, its image is generated by $[G_1] = \rho_0(l)$ and $\rho_0(m)=[\mathrm{Id}]$.

 In order to check the second hypothesis, we use the results of \cite{deraux_uniformizations}. In that paper, Deraux finds with a different technique the Deraux-Falbel structure of \cite{falbel}. He considers a Ford domain $F$ in $\h2c$ for $\Gamma = \rho_0(\pi_1(M))$ (theorem 5.1) and then studies its boundary at infinity in $\dh2c$ (section 8). $M$ is then obtained as a quotient of a $G_1$-invariant domain $E= \partial_{\infty}F$, that is, in $\dh2c \simeq (\mathbb{C}\times \mathbb{R}) \cup \{\infty\}$, the exterior of a $G_1$-invariant cylinder $C$ embedded in $\mathbb{C}\times \mathbb{R}$ (proposition 8.1). $E$ is a $[G_1]$-horotube; so there exists $s\in [0,1[$ such that the image by the developing map of $\widetilde{T}_{[s,1[}$ is a $[G_1]$-horotube contained in $E$. Thus, the second hypothesis is satisfied.
 
 So, the conclusion of theorem  \ref{thm_chirurgie} holds. By changing coordinates in order to have the usual marking for the fundamental group of the boundary of $M$, we get:

\begin{prop}\label{prop_chirurgie_8}
  There is an open set $\Omega$ of $\mathcal{R}_1(\pi_1(M), \pu21)$ such that, for all $\rho \in \Omega$ such that the image of $h_\rho$ is generated by an element $[U] \in \pu21$, there exists a  \CR {} structure on $M$ of holonomy $\rho$. Furthermore, for the usual marking $(l_0,m_0)$ of $\pi_1(T)$,
 
 \begin{enumerate}
  \item If $[U]$ is loxodromic, the structure extends to a \CR {} structure on the Dehn surgery of type $(-1,3)$ of $M$.
  \item If $[U]$ is elliptic of type $(\frac{p}{n},\frac{\pm 1}{n})$, the structure extends to a \CR {} structure on the Dehn surgery of type $(-n,\pm p+3n)$ of $M$.
  \item If $[U]$ is elliptic of type $(\frac{p}{n},\frac{q}{n})$ with $|p|,|q|>1$, the structure extends to a \CR {} structure on the gluing of $M$ and a compact manifold with torus boundary $V(p,q,n)$ through their boundaries. Furthermore, $V(p,q,n)$ is the complement of a torus knot in the lens space $L(n,\alpha)$ where $\alpha \equiv p^{-1}q \mod n$.
 \end{enumerate}
\end{prop} 

\begin{rem}
 If $[U]$ is parabolic, then the theorem still holds, but the \CR {} structure extends to a thickening of $M$. These structures are the structures given by Deraux in \cite{deraux_uniformizations}.
\end{rem}

\subsection{Deformations of the structure} 
 It remains to see that the open set $\Omega \subset \mathcal{R}_1(\pi_1(M))$ is not reduced to a point to get interesting conclusions. The representation $\rho_0$ is in the component $R_2$ of the $\mathrm{SL}_3(\mathbb{C})$-character variety described in \cite{character_sl3c} by Falbel, Guilloux, Koseleff, Rouillier and Thistlethwaite. In section 5 of that paper, the representations in $R_2$ taking values in $\su21$ are parametrized up to conjugacy, at least in a neighborhood of $\rho_0$, by a complex parameter $u=\mathrm{tr} (\rho(m_0))$. We denote by $G(u)= \rho(m_0)$ the corresponding matrix.

Setting $v=\con{u}$ and $\Delta = 4u^3 + 4 v^3 -u^2v^2 -16uv+16$, the parametrization is explicitly given by:

 \[ [G_3^{-1}(u)] = \rho(a)= \begin{bmatrix}
 \frac{v}{2} & 1 & -\frac{(1-i)(-16+8uv-2v^3-4\sqrt{\Delta})}{8u^2-6uv^2+v^4}  \\
  \frac{1}{8}(1+i)(-2u+v^2) & \frac{1}{4}(1+i)v & 1  \\
   \frac{1}{16}(8-4uv+v^3-2\sqrt{\Delta}) & \frac{1}{8}(-4u + v^2) & \frac{1}{4}(1-i)v  \\
 \end{bmatrix} \]
 
\[ [G_1^{-1}(u)] =
\rho(b)=\begin{bmatrix}
 \frac{v}{2} & i & \frac{(1+i)(-16+8uv-2v^3-4\sqrt{\Delta})}{8u^2-6uv^2+v^4}  \\
  -\frac{1}{8}(1+i)(-2u+v^2) & \frac{1}{4}(1-i)v & i  \\
   -\frac{1}{16}(8-4uv+v^3-2\sqrt{\Delta}) & -\frac{i}{8}(-4u + v^2) & \frac{1}{4}(1+i)v  \\
 \end{bmatrix}
 \] 
 
 Recall that, for these choice of generators the usual meridian $m_0$ is given by $m_0=a^{-1}$. The Hermitian form preserved by this representation is given by the matrix\footnote{We write here the opposite of the matrix $H$ of \cite{character_sl3c} in order to have signature $(2,1)$ and not $(1,2)$.}
 
\[H = \begin{pmatrix}
 \frac{1}{8}(\Delta -16)(\sqrt{\Delta}+|u|^2-4) & 0 & 0  \\
  0 & 16 - \Delta & 0  \\
   0 & 0 & 8(\sqrt{\Delta}+4)  \\
 \end{pmatrix}.\]

  Furthermore, in the whole component the relation $\rho(l_0)=\rho(m_0)^3$ holds, so $\mathcal{R}_1(\pi_1(M)) \cap R_2 = \mathcal{R}(\pi_1(M)) \cap R_2$. By projecting on $\pu21$, we can apply theorem \ref{thm_chirurgie} on an open set containing $3=\mathrm{tr}(\rho_0(m_0))$ with these parameters.

Figure \ref{curve_char}, taken from \cite{character_sl3c} shows an open set of $\mathbb{C}$ where we have representations. By noting $\mathrm{tr}(\rho_0(m_0)) = x+iy$, the component containing $\rho_0$ admits as parameter the regions with boundary the curve $\Delta(x,y)=0$ and containing the points $3$, $3\omega$ and $3\omega^2$, where
$\Delta(x,y)=-x^4 -y^4 - 2 x^2 y^2 - 24 x y^2 + 8 x^3 - 16 x^2 - 16 y^2 + 16$.

\begin{figure}[H]
\center
 \includegraphics[width=6cm]{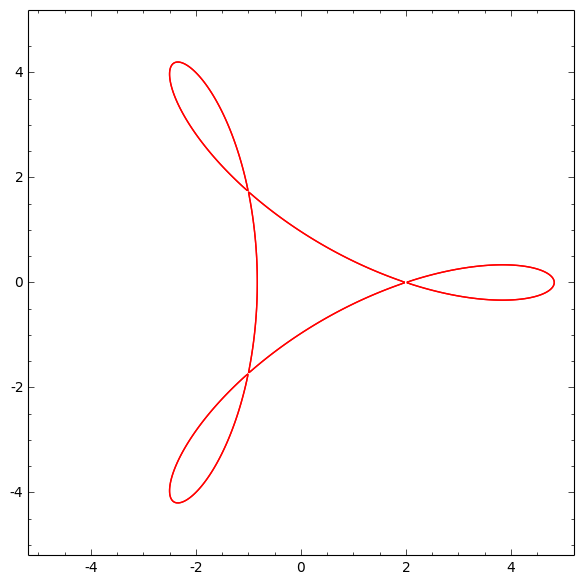}
 \caption{Domain parametrizing a component of the deformation variety near $\rho_0$.}\label{curve_char}
\end{figure}  

Now let us plot the curve $\mathcal{C}$ of traces of non-regular elements of $\su21$. It is given by the zeroes of the function $f(z)=|z|^4-8\Re (z^3) + 18|z|^2-27$ (see proposition \ref{fonction_goldman}). The curve separates regular elliptic and loxodromic elements. It has a singularity at the point $u=3$: thus a neighborhood of this point contains points corresponding to representations where the peripheral holonomy is loxodromic and points where it is regular elliptic. 

\begin{figure}[H]
\center
 \includegraphics[width=6cm]{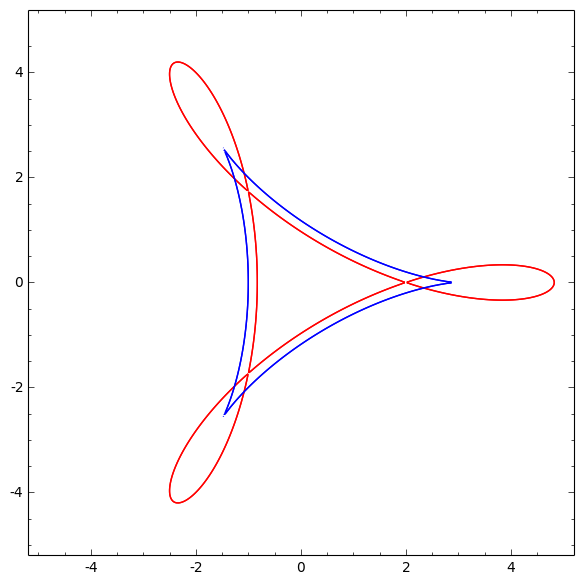}
 \includegraphics[width=6cm]{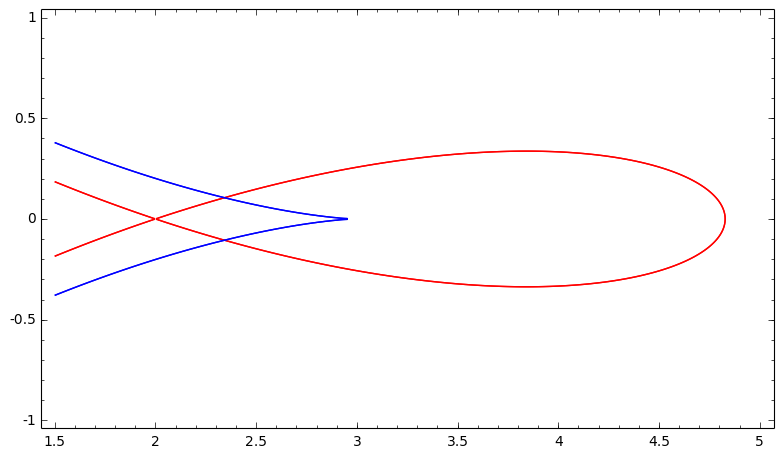}
 \caption{Curve of non-regular elements in a component of the deformation variety near $\rho_0$.}\label{curve_char2}
\end{figure}  

\begin{rem}
 The parabolic deformations of the Deraux-Falbel structure given by Deraux in \cite{deraux_uniformizations} correspond to the points of $\mathcal{C}$.
\end{rem}

We can therefore apply the first point of proposition \ref{prop_chirurgie_8} to the space of holonomy representations given by the parameters above. We obtain the following proposition:

\begin{prop}
 There exist infinitely many  \CR {} structures on the Dehn surgery of $M$ of type $(-1,3)$.
\end{prop}

\begin{rem}
 This surgery is the unit tangent bundle to the hyperbolic orbifold $(3,3,4)$. It is a Seifert manifold of type $S^2(3,3,4)$. See for example chapter 5 of the book of Cooper, Hodgson and Kerckhoff \cite{cooper2000three} or the paper of Deraux \cite{deraux2014spherical}. Deraux also remarks in \cite{deraux_uniformizations} (section 4) and \cite{deraux2014spherical} (theorem 4.2), that the image of $\rho_0$ is a faithful representation of the even words of the $(3,3,4)$ triangle group, generated by involutions $I_1,I_2,I_3$. This identification satisfies $G_1=I_2 I_3 I_2 I_1 $, $G_2=I_1 I_2$, $G_3=I_2 I_1 I_2 I_3 $ and the triangle group relations: $(G_2)^4 = (I_1 I_2)^4 = \mathrm{Id}$, $(G_1 G_2)^3 = (I_2 I_3)^3=\mathrm{Id}$ and $(G_2 G_1 G_2)^3 = (I_1 I_3)^3=\mathrm{Id}$. Furtheremore, the image of the usual meridian $m_0$ is $G_3$.
 
 This group is the fundamental group of a Seifert manifold of type $S^2(3,3,4)$. Since the relation $l_0=m_{0}^3$ holds in the whole component $R_2$, the images of representations in $R_2$ are representations of this index two subgroup of the $(3,3,4)$ triangle group. Furthermore, Parker, Wang and Xie show in \cite{parker_wang_xie} that a $\pu21$ representation of the $(3,3,4)$ triangle group is discrete and faithful if and only if the image of $I_1 I_3 I_2 I_3 $ is nonelliptic. Note that $G_1 I_1 I_3 I_2 I_3 = (I_2 I_3)^3 = \mathrm{Id}$, so the representation of the triangle group is discrete and faithful if and only if the corresponding peripheral holonomy is nonelliptic. They also give a one parameter family of such representations, corresponding to parameters $u \in \mathbb{R}$. Thus, there exists $\delta > 0$ such that all the \CR {} structures on the Dehn surgery of $M$ of type $(-1,3)$ with parameter $u$ in the interval $]3,3+\delta[$ have discrete and faithful holonomy.
\end{rem}

Since the parameter is the trace of an element, we know that cases 2 and 3 happen infinitely many times, but we can not distinguish at first sight, for a given trace, if it is a Dehn surgery or a gluing of a $V(p,q,n)$ manifold. Nevertheless, a computation with the explicit parametrization of \cite{character_sl3c} and the continuity of eigenvalues we prove:

\begin{prop}
 There is $\delta > 0$ such that, if $p,n \in \mathbb{N}$ are integers relatively prime such that $\frac{p}{n} < \delta$, then the Dehn surgery of M of type $(-n , -p + 3n)$ admits a \CR {} structure.
\end{prop}

\begin{proof}
 Let $p,n \in \mathbb{N}$ be integers relatively prime. Let $\alpha = \frac{-2p-1}{3n}$, $\beta = \frac{2+p}{3n}$, $\gamma = \frac{p-1}{3n}$ and $u=e^{i\alpha} + e^{i\beta} + e^{i\gamma}$. We only need to show that if $\frac{p}{n}$ is small enough, then the eigenvalue of $\rho(m)=G_3^{-1}(u)$ corresponding to a negative eigenvector is $e^{i\gamma}$, and so $G_3(u)$ will be of type $(\frac{p}{n},\frac{-1}{n})$.
 
 Since eigenvectors and eigenvalues are continuous functions of $u$ in the connected component of regular elliptics, in $R_2$, as in figure \ref{curve_char2},   the statement is true for all $(p,n)$ if and only if it is true for a particular choice of $(p,n)$. For the arbitrary choice $(p,n)=(3,23)$ a computation shows that $G_3(u)$ is of type $(\frac{3}{23},\frac{-1}{23})$.
\end{proof}

\section{ Proof of theorem \ref{thm_chirurgie}.} \label{section_proof}

 In this section, we are going to prove theorem \ref{thm_chirurgie}. We use notation of section \ref{section_surgeries}. We then have a manifold $M$ with a torus boundary $T$, endowed with a \CR {} structure $(\Dev_0,\rho_0)$ of unipotent peripheral holonomy $h_{\rho_0}$  of rank 1 and generated by an element $[U_0] \in \pu21$. We suppose that there is $s \in [0,1[$ such that $\Dev_0(\widetilde{T}_{[s,1[})$ is a $[U_0]$-horotube. Recall that we work with a single boundary component $T$ to avoid heavy notation. The proof works for several boundary components.
 
 In order to prove the theorem, we begin by rewriting the hypotheses to make them easier to handle. The existence of a \CR {} structure on $M$ for a deformation of $\rho_0$ will be a consequence of the Ehresmann-Thurston principle. To extend it to a surgery of $M$, we need only a local surgery result by looking near the boundary of $M_{[0,1[}$. This surgery result is very similar, in cases 1 and 2, to the one given by Schwartz in chapter 8 of \cite{schwartz}.

\subsubsection{Rewritting the hypotheses}\label{nouvelles_hyp}

First of all, we rewrite the second hypothesis.  
Fix a diffeomorphism $\psi : \mathbb{R}^2 \times [0,1[ \rightarrow \widetilde{T}_{[0,1[}$, such that:
 \begin{enumerate}
  \item For all $s \in [0,1[$ $\psi (\mathbb{R}^2 \times \{s\}) = \widetilde{T}_s$.
  \item $\psi$ induces a diffeomorphism between  $\mathbb{R} \times S^1 \times [0,1[ $ and $ \widetilde{T}_{[0,1[} / \ker (h_{\rho_0})$.
 \end{enumerate}
 To avoid too much notation, we identify $\mathbb{R}^2 \times [0,1[$ with $\widetilde{T}_{[0,1[}$ , and also $\mathbb{R} \times S^1 \times [0,1[ $ to $ \widetilde{T}_{[0,1[} / \ker (h_{\rho_0})$.
  In this case, the developing map $\Dev_0$ induces a diffeomorphism between $ \widetilde{T}_{[0,1[} / \ker(h_{\rho_0})$ and $\Dev_0( \widetilde{T}_{[0,1[})$, that we will still call $\Dev_0$.
  We change hypothesis (2) of the theorem for hypotheses (2') and (3) described below :
 
 \paragraph{ Hypothesis (2'): } There are $0<s_1<s_2<1$ such that
  \begin{enumerate}
   \item For all $s \in [s_1,s_2]$, $\Dev_0( \{0\} \times S^1 \times \{s\})$ is a circle transverse to the flow.
   \item For all $(t,\zeta,s) \in \mathbb{R} \times S^1 \times [s_1,s_2]$, $\Dev_0(t,\zeta,s)=\phi_t^{\rho_0}(\Dev_0(0,\zeta,s))$.
  \end{enumerate}\begin{figure}[H]
\center
 \includegraphics[width=10cm]{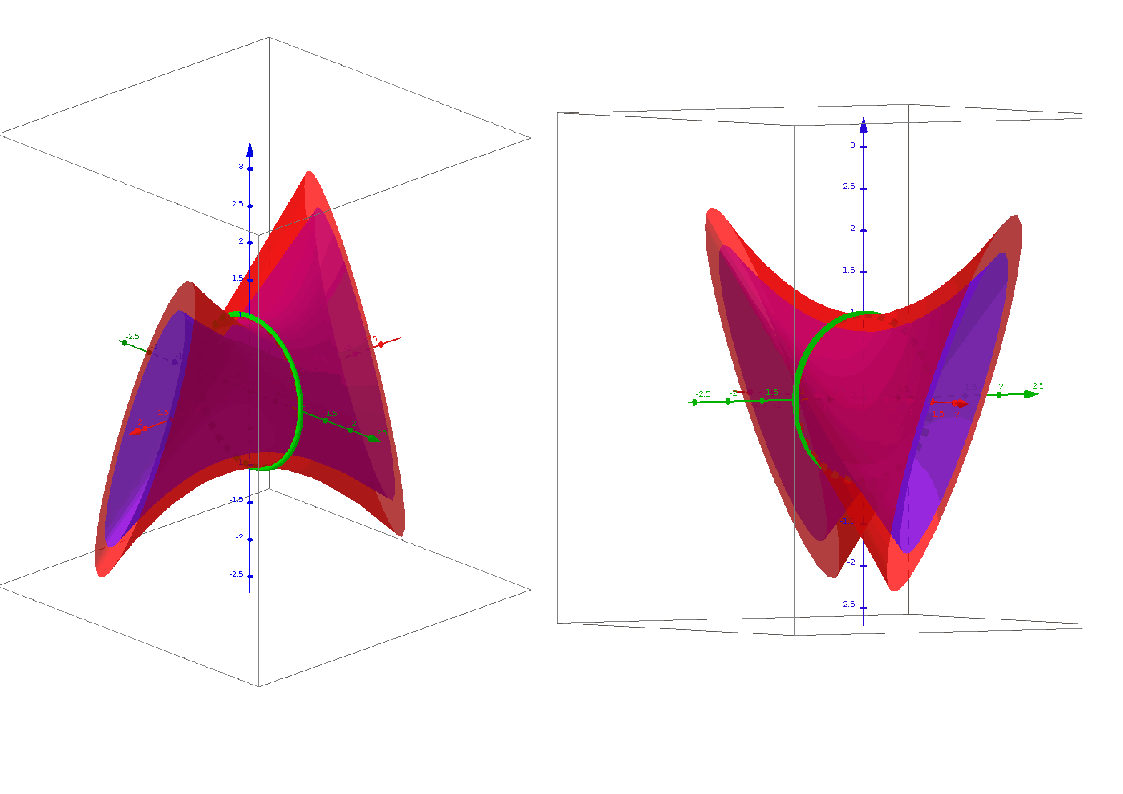}
 \caption{Two views of surfaces bounding a region of the form $\Dev_0( \widetilde{T}_{[s_1,s_2]})$}\label{horotube34}
\end{figure}

\begin{rem}
 Thanks to remark \ref{rq_horotube_gentil}, it is clear that hypothesis 2 implies hypothesis (2'). Perhaps after considering an isotopy and increasing slightly $s$, we can suppose that the horotube $\Dev_0(\widetilde{T}_{[s,1[})$ is nice. We only need then to consider the restriction to a segment $\widetilde{T}_{[s_1,s_2]}$.
\end{rem}

Hypothesis (2') gives, in particular, that $\Dev_0(\widetilde{T}_{s_2})$ separates $\dh2c \setminus \{p\}$ in two connected components: a solid cylinder $C_{s_2}$ and the exterior of this cylinder, which is homeomorphic to $S^1 \times \mathbb{R} \times ]0, + \infty[$. Hypothesis (3) tells us that the structure of $M$ is in the correct side of the tube:

\paragraph{Hypothesis (3) :} $\Dev_0(\widetilde{T}_{s_1})$ is contained in $C_{s_2}$.

\begin{rem}
 Hypothesis (2) is equivalent to hypotheses (2') and (3). The implication from (2) to (2') and (3) is clear, and, if we suppose (2') and (3), the structure can be extended to the outside in such a way that $\Dev_0(\widetilde{T}_{[s_2,1[})$ is the horotube of boundary $\Dev_0(\widetilde{T}_{s_2})$.
\end{rem}

\subsubsection{Deforming the structure}
 We now prove Theorem \ref{thm_chirurgie}. Assume the hypotheses of section \ref{nouvelles_hyp} are satisfied. Let $\rho$ be a deformation close to $\rho_0$ in $\mathcal{R}_1(\pi_1(M), \pu21)$ such that $h_\rho(m) = \mathrm{Id}$. The image of $h_\rho$ is then generated by $[U]=\rho(l)$. We suppose that $[U]$ is a regular element.

 Let $\epsilon > 0$. By the Ehresmann-Thurston principle, if $\rho$ is close enough to $\rho_0$, there is a \CR {} structure on $M_{[0,s_2 + \epsilon[}$ with holonomy map $\rho$. We have then a developping map $\Dev_{\rho} : \widetilde{M}_{[0,s_2 + \epsilon[} \rightarrow \dh2c$ close to $\Dev_0$ in the $\mathcal{C}^1$ topology. So, we can suppose that $\Dev_\rho$ is still a diffeomorphism between the compact set $[-\epsilon,1 + \epsilon] \times S^1 \times [s_1,s_2]$ and its image.
 
 \begin{rem}
 We have then an atlas of charts of $T_{[s_1,s_2]}$ taking values in $\dh2c$ by choosing lifts of $T_{[s_1,s_2]}$ in $[-\epsilon,1 + \epsilon] \times S^1 \times [s_1,s_2] \subset \widetilde{T}_{[s_1,s_2]}$. Transition maps are given by powers of $[U] = \rho(l)$.
 \end{rem}

 Fix $s_1 < s'_1 < s'_2 < s_2$.

\begin{lemme}[Straightening]\label{lemme_redressement}
 If $\rho$ is close enough to $\rho_0$, perhaps after taking an isotopy of $\Dev_\rho$,  we have, for all $(t,\zeta,s) \in \mathbb{R} \times S^1 \times [s'_1,s'_2]$, that $\Dev_\rho(t,\zeta,s) = \phi_t^\rho( \Dev_\rho(0,\zeta,s))$.
\end{lemme}
 \begin{proof}
  The flows $\phi_t^\rho$ and $\phi_t^{\rho_0}$ are close in the $\mathcal{C}^1$ topology when $\rho$ is close to $\rho_0$. We deduce that the deformation from $\rho_0$ to $\rho$ induces a  $\mathcal{C}^1$ deformation from $\phi_t^{\rho_0} \circ \Dev_0$ to $\phi_t^\rho \circ \Dev_\rho$.
  First we restrict to the compact set $[0,1] \times S^1 \times [s'_1,s'_2]$, which is in the interior of $[-\epsilon,1 + \epsilon] \times S^1 \times [s_1,s_2]$.
  
  Since $$\Dev_0([0,1] \times S^1 \times [s'_1,s'_2]) = \bigcup_{t \in [0,1]} \phi_t^{\rho_0}(\{0\}\times S^1 \times [s'_1,s'_2]),$$ if $\rho$ is close enough to $\rho_0$ then  $$\bigcup_{t \in [0,1]} \phi_t^{\rho}(\{0\}\times S^1 \times [s'_1,s'_2])$$ is contained in the interior of $\Dev_\rho([0,1] \times S^1 \times [s_1,s_2])$.
  
  Since $[U] \cdot \phi_t^\rho = \phi_{t+1}^\rho$ and $[U] \cdot \Dev_\rho (t,\zeta,s) = \Dev_\rho(t+1,\zeta,s) $, we can straighten $\Dev_\rho$ by a  $[U]$-equivariant isotopy to have, for $(t,\zeta,s) \in \mathbb{R} \times S^1 \times [s'_1,s'_2]$ that $\Dev_\rho(t,\zeta,s) = \phi_t^\rho( \Dev_\rho(0,\zeta,s))$.
 \end{proof}

From now on, we suppose that for all $(t,\zeta,s) \in \mathbb{R} \times S^1 \times [s'_1,s'_2]$ we have $\Dev_\rho(t,\zeta,s) = \phi_t^\rho( \Dev_\rho(0,\zeta,s))$.

\begin{lemme}\label{lemme_pas_enlace}
 Let $C$ be a $\mathbb{C}$-circle invariant by $[U]$. Then $C$ and the annulus $\Dev_\rho(\{0\} \times S^1 \times [s'_1,s'_2])$ are not linked.
\end{lemme}
 \begin{proof}
 $[U]$ is a regular element close to the unipotent element $[U_0]$, which has fixed point $p_0 \in \dh2c$. Thanks to remarks \ref{rem_axes_lox_loin} and \ref{rem_axes_ell_loin}, we know that $C$ leaves every compact of $\dh2c \setminus \{p_0\}$ when $[U]$ approaches $[U_0]$. Since $\Dev_\rho(\{0\} \times S^1 \times [s'_1,s'_2])$ stays in a fixed compact set when we deform $\rho_0$ to $\rho$, we deduce that $C$ and the annulus $\Dev_\rho(\{0\} \times S^1 \times [s'_1,s'_2])$ are not linked.
 \end{proof}

 It only remains to establish a local surgery result. It is essentially what Schwartz does in chapter 8 of \cite{schwartz}.
 
 Thanks to lemma \ref{lemme_redressement}, we know that $\Dev_\rho(\widetilde{T}_{[s'_1,s'_2-\epsilon]} )$ is the orbit by $\phi_t^\rho$ of the annulus  $A = \Dev_\rho(\{0\} \times S^1 \times [s'_1,s'_2 - \epsilon])$. This orbit separates $\dh2c$ (if $[U]$ est elliptic) or $\dh2c$ minus two points (if $[U]$ is loxodromic), in two connected components $C_1$ and $C_2$, of respective boundaries $\Dev_\rho(\widetilde{T}_{s'_1})$ and $\Dev_\rho(\widetilde{T}_{s'_2})$. We have a proper action of $[U]$ on $C_2$, and so we can consider the quotient manifold $N = C_2 / \langle [U] \rangle$. It is a compact manifold with a torus boundary, endowed with a \CR {} structure which coincides with the structure of $M_{[0,s'_2[}$ on $T_{]s'_2 - \epsilon , s'_2[}$. Thus, the gluing $M_{[0,s'_2[} \cup N / \sim$ has a \CR {} structure which extends the structure $(\Dev_\rho, \rho)$ of $M$.

 We are going to show that if $[U]$ is loxodromic or elliptic of type $(\frac{p}{n},\frac{1}{n})$, then $N$ is a solid torus and that we have a  \CR {} structure on a Dehn surgery of $M$ of a certain slope. If $[U]$ is elliptic of type $(\frac{p}{n},\frac{q}{n})$, we will see a description of $N$ as a complement of a torus knot in some lens space.


\paragraph{Case 1 : $[U]$ is loxodromic}
 We work in Siegel's model, and we identify $\dh2c$ to $(\mathbb{C} \times \mathbb{R}) \cup \{ \infty \}$. Perhaps after a conjugation, we can suppose that 
 \[U = T_{\lambda} = \begin{pmatrix}
\lambda & 0 & 0 \\
0 & \frac{\overline{\lambda}}{\lambda} & 0 \\
0 & 0 &  \frac{1}{\overline{\lambda}} \\
\end{pmatrix} .\]

$[U]$ has two fixed points: $(0,0)$ and $\infty$. Let $S$ be the sphere centered at $(0,0)$ and of radius 1 in $\mathbb{C} \times \mathbb{R}$. This sphere is a fundamental domain for the action of the flow $\phi_t^\rho$.
 The subgroup generated by $[U]$ acts properly on $(\mathbb{C} \times \mathbb{R}) \setminus (0,0)$, and the region $\bigcup_{t \in [0,1]} \phi_t^\rho(S)$, of boundaries $S$ and $[U] \cdot S$ is a fundamental domain for this action. 
The orbit of $A$ under $\phi_t^\rho$ intersects $S$ in an annulus that separates $S$ in two disks $D_1$ and $D_2$, so that their orbits under $\phi_t$ are the connected components $C_1$ and $C_2$ respectively. Figure \ref{domaine_ch_lox} shows this situation.

\begin{figure}[H]
\center
 \includegraphics[width=6cm]{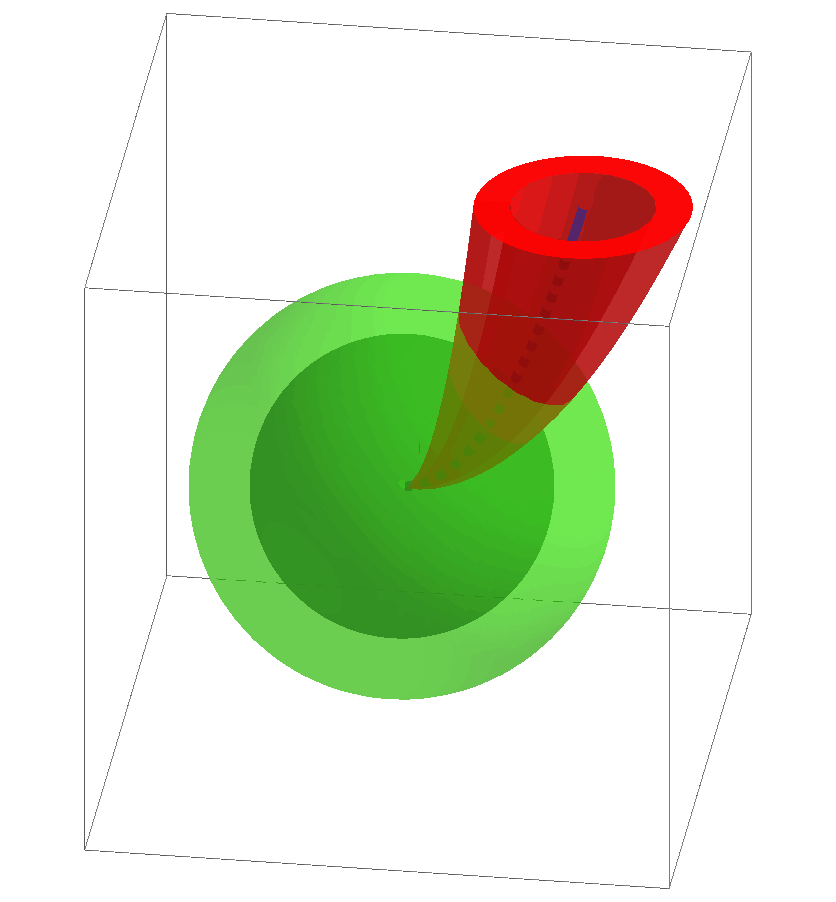}
 \caption{The orbit of $A$ under $\phi_t$ (in red) and the spheres $S$ and $[U] \cdot S$ (in green).}\label{domaine_ch_lox}
\end{figure}

 The quotient manifold $N = C_2 / \langle [U] \rangle$ is obtained by identifying $D_2$ and $[U] \cdot D_2$ in $\bigcup_{t \in [0,1]} \phi_t^\rho(D_2)$. Thus, it is a solid torus. But the curve of $\pi_1(T)$ that becomes trivial in $C_2$ is the one homotopic to the boundary of $D_2$ : so it is $m$. We deduce that the surgery is of type $(0,1)$.

\paragraph{Case 2 : $[U]$ is elliptic of type $(\frac{p}{n},\frac{\pm 1}{n})$.}

By choosing $[U_0]^{\pm 1}$ instead of $[U_0]$ as the generator of the peripheral holonomy, we can suppose that $U$ is of type $(\frac{\pm p}{n},\frac{1}{n})$. For ease of exposition, we write the proof for $[U]$ of type $(\frac{p}{n},\frac{1}{n})$.

We reason in the same way as in the loxodromic case. Thanks to lemma \ref{lemme_pas_enlace}, we know that $\Dev_\rho(\widetilde{T}_{[s'_1,s'_2-\epsilon]} )$ is the orbit under $\phi_t$ of the annulus  $A = \Dev_\rho(\{0\} \times S^1 \times [s'_1,s'_2])$, which is not linked to any of the invariant $\mathbb{C}$-circles of $[U]$.

 The orbit of $A$ under the flow $\phi_t^\rho$ is then homeomorphic to $S^1 \times S^1 \times [s'_1,s'_2]$. Its complement in $\dh2c$ has two connected components. Let $C_2$ be the component with boundary $\Dev_\rho(\widetilde{T}_{s'_2})$. Following remark \ref{rem_noeud_tor}, the orbits of the flow are not knotted: the two connected components are solid tori, and $[U]$ acts properly on each one.  But the quotient of a solid torus by a proper action of a finite group is still a solid torus. The quotient manifold $N = C_2 / \langle [U] \rangle$ is then a solid torus, and we have a  \CR {} structure on a Dehn surgery of $M$. It only remains to identify it.

 Perhaps after a conjugation, we can suppose that
  \[U = e^{\frac{-2i\pi (p+1)}{3n}} \begin{pmatrix}
e^{\frac{2i\pi p}{n}} & 0 & 0 \\
0 & e^{\frac{2i\pi}{n}} & 0 \\
0 & 0 &  1 \\
\end{pmatrix} .\] in the ball model.
 In Siegel's model, by identifying $\dh2c$ to $(\mathbb{C}\times \mathbb{R}) \cup \{\infty\}$, we have that $[CUC^{-1}]$ stabilizes two $\mathbb{C}$-circles: the circle $\mathcal{C}_1$ centered at $0$ of radius $\sqrt{2}$ in $\mathbb{C}\times \{0\}$ and $\mathcal{C}_2$, the axis $\{0\} \times \mathbb{R}$. A generic orbit of the flow turns one time around $\mathcal{C}_1$ and $p$ times around $\mathcal{C}_2$.
 
 Let $\gamma$ be the loop that follows the $\mathbb{C}$-circle $\mathcal{C}_2$ and is oriented so that the meridian $m$ is homotopic to $\gamma$ in the component $C_2$. In this case, $nl$ is homotopic, also in $C_2$, to $-p \gamma$. Thus $nl+pm$ is a homotopically trivial loop in $C_2$, which is a covering of the solid torus $N$ glued to $M$. So it is also a trivial loop in $N$. We deduce that the surgery is of type $(n,p)$.
 
\begin{figure}[H]
\center
 \includegraphics[width=6cm]{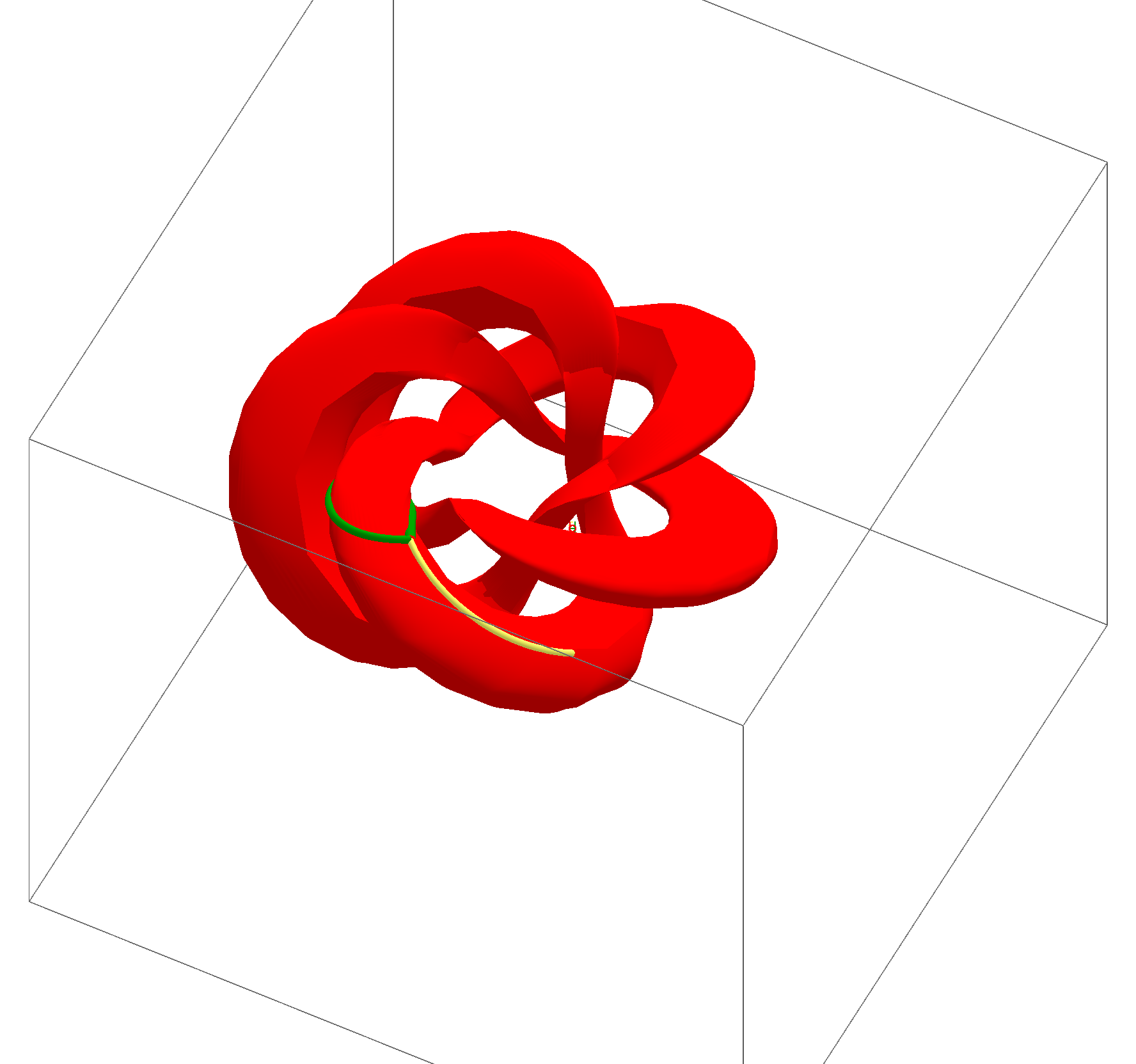}
 \caption{The orbit of $A$ under $\phi_t$ (in red), the longitude $l$ (in yellow) and the meridian $m$ (in green).}\label{orbite_ell3}
\end{figure}

 \paragraph{Case 3 : $[U]$ is elliptic of type $(\frac{p}{n},\frac{q}{n})$.}
 
 The idea is the same as in cases 1 and 2. We know that $\Dev_\rho(\widetilde{T}_{[s'_1,s'_2-\epsilon]} )$ is the orbit by $\phi_t$ of the annulus  $A = \Dev_\rho(\{0\} \times S^1 \times [s'_1,s'_2])$, which is not linked to any of the invariant $\mathbb{C}$-circles of $[U]$.

 The orbit of $A$ under the flow $\phi_t^\rho$ is homeomorphic to $S^1 \times S^1 \times [s'_1,s'_2]$. Its complement in $\dh2c$ has two connected components. Let $C_2$ be (again) the component of boundary $\Dev_\rho(\widetilde{T}_{s'_2})$ and $C_1$ the one of boundary $\Dev_\rho(\widetilde{T}_{s'_1})$. According to remark \ref{rem_noeud_tor}, generic orbits of the flow are torus knots of type $(p,q)$: $C_1$ is then a tubular neighborhood of one of the orbits and $C_2$ is homeomorphic to the complement of a torus knot of type $(p,q)$. But $[U]$ acts properly on $\dh2c$ and stabilizes $C_1$ and $C_2$.
 
 Remark that, in the ball model, the action of the group generated by $[U]$ is the same as the one of the group generated by $(z_1,z_2) \mapsto (e^{\frac{2i\pi}{n}}z_1,e^{\frac{2i\pi \alpha}{n}}z_2)$ where $\alpha \equiv p^{-1}q \mod n$. The quotient $\dh2c / \langle [U] \rangle$ is then homeomorphic to the lens space $L(n,\alpha)$. Furthermore, $C_1 / \langle [U] \rangle$ is a solid torus knotted in $\dh2c / \langle [U] \rangle$. The quotient manifold $V(p,q,n) = C_2 / \langle [U] \rangle$ is the complement of a torus knot in $\dh2c / \langle [U] \rangle \simeq L(n,\alpha)$. The  \CR {} structure of $M$ extends then to the gluing of $M$ and $V(p,q,n)$ through their torus boundary components.

 \bibliographystyle{alpha}
\bibliography{./tex/ref_cr}

\newcommand{\etalchar}[1]{$^{#1}$}
\begin{thebibliography}{FGK{\etalchar{+}}14}

\bibitem[BG04]{bergeron_gelander}
N.~Bergeron and T.~Gelander.
\newblock A note on local rigidity.
\newblock {\em Geom. Dedicata}, 107:111--131, 2004.

\bibitem[CHK00]{cooper2000three}
Daryl Cooper, Craig~David Hodgson, and Steve Kerckhoff.
\newblock {\em Three-dimensional orbifolds and cone-manifolds}, volume~5.
\newblock Mathematical society of Japan, 2000.

\bibitem[Der14a]{deraux_uniformizations}
Martin Deraux.
\newblock A 1-parameter family of spherical {CR} uniformizations of the figure
  eight knot complement.
\newblock {\em arXiv preprint arXiv:1410.1198}, 2014.

\bibitem[Der14b]{deraux2014spherical}
Martin Deraux.
\newblock On spherical {CR} uniformization of 3-manifolds.
\newblock {\em arXiv preprint arXiv:1410.0659}, 2014.

\bibitem[DF{\etalchar{+}}13]{falbel}
Martin Deraux, Elisha Falbel, et~al.
\newblock Complex hyperbolic geometry of the figure eight knot.
\newblock 2013.

\bibitem[FGK{\etalchar{+}}14]{character_sl3c}
Elisha Falbel, Antonin Guilloux, Pierre-Vincent Koseleff, Fabrice Rouillier,
  and Morwen Thistlethwaite.
\newblock Character varieties for $\mathrm{SL} (3, \mathbb{C})$: the figure
  eight knot.
\newblock {\em arXiv preprint arXiv:1412.4711}, 2014.

\bibitem[Gen10]{genzmer}
Juliette Genzmer.
\newblock {\em Sur les triangulations des structures
  $\mathrm{CR}$-sphériques}.
\newblock PhD thesis, UPMC, 2010.

\bibitem[Gol99]{goldman}
William~Mark Goldman.
\newblock {\em Complex hyperbolic geometry}.
\newblock Oxford University Press, 1999.

\bibitem[Gol10]{goldman_manifolds}
William~M. Goldman.
\newblock Locally homogeneous geometric manifolds.
\newblock In {\em Proceedings of the {I}nternational {C}ongress of
  {M}athematicians. {V}olume {II}}, pages 717--744. Hindustan Book Agency, New
  Delhi, 2010.

\bibitem[PWX]{parker_wang_xie}
John~R Parker, Jieyan Wang, and Baohua Xie.
\newblock Complex {H}yperbolic (3, 3, n) triangle groups.
\newblock {\em To appear in Pacific Journal of Mathematics}.

\bibitem[Sch07]{schwartz}
Richard~Evan Schwartz.
\newblock {\em Spherical {CR} Geometry and Dehn Surgery (AM-165)}.
\newblock Number 165. Princeton University Press, 2007.

\bibitem[TM79]{gt3m}
William~P Thurston and John~Willard Milnor.
\newblock {\em The geometry and topology of three-manifolds}.
\newblock Princeton University Princeton, 1979.

\end{thebibliography}

\end{document}